\documentclass{article} 

\usepackage{ifsym}

\usepackage[super]{nth}
\usepackage{amsmath}
\usepackage{amssymb}
\usepackage{amsthm}
\usepackage[latin1]{inputenc}
\usepackage{pxfonts}
\usepackage{txfonts}

\usepackage{graphicx}

\usepackage[usenames,dvipsnames]{pstricks}
\usepackage{epsfig}
\usepackage{pst-grad} 
\usepackage{pst-plot} 
 \usepackage[usenames,dvipsnames]{pstricks}
 \usepackage{epsfig}
 \usepackage{pst-grad} 
 \usepackage{pst-plot} 
 \usepackage[space]{grffile} 
 \usepackage{etoolbox} 
 \makeatletter 
 \patchcmd\Gread@eps{\@inputcheck#1 }{\@inputcheck"#1"\relax}{}{}
 \makeatother

\usepackage{ifthen}
\usepackage{color}

\newcommand{\intav}[1]{\mathchoice {\mathop{\vrule width 6pt height 3 pt depth  -2.5pt
\kern -8pt \intop}\nolimits_{\kern -6pt#1}} {\mathop{\vrule width
5pt height 3  pt depth -2.6pt \kern -6pt \intop}\nolimits_{#1}}
{\mathop{\vrule width 5pt height 3 pt depth -2.6pt \kern -6pt
\intop}\nolimits_{#1}} {\mathop{\vrule width 5pt height 3 pt depth
-2.6pt \kern -6pt \intop}\nolimits_{#1}}}

\def\polhk#1{\setbox0=\hbox{#1}{\ooalign{\hidewidth\lower1.5ex\hbox{`}\hidewidth\crcr\unhbox0}}}

 \newcommand{\Rr}{\mathbb R}

\newcommand{\lipop}{\operatorname{Lip}}
\newcommand{\llip}{\operatorname{Log-Lip}}

\newtheorem{teo}{Theorem}[section]
\newtheorem{taggedtheoremx}{A}
\newenvironment{taggedtheorem}[1]
 {\renewcommand\thetaggedtheoremx{#1}\taggedtheoremx}
 {\endtaggedtheoremx}
\newtheorem{Definition}{Definition}[section]

\newtheorem{Lemma}{Lemma}[section]
\newtheorem{Corollary}{Corollary}[section]
\newtheorem{Proposition}{Proposition}[section]
\newtheorem{Remark}{Remark}[section]
\newtheorem{Assumption}{A}

\begin{document}

\title{Interior Sobolev regularity for fully nonlinear parabolic equations}
\author{Ricardo Castillo
\,\,and \,\,
Edgard A. Pimentel
}

\date{\today} 

\maketitle

\begin{abstract}

\noindent In the present paper, we establish sharp Sobolev estimates for solutions of fully nonlinear parabolic equations, under minimal, asymptotic, assumptions on the governing operator. In particular, we prove that solutions are in $W^{2,1;p}_{loc}$. Our argument unfolds by importing improved regularity from a limiting configuration. In this concrete case, we recur to the recession function associated with $F$. This machinery allows us to impose conditions solely on the original operator at the infinity of $\mathcal{S}(d)$. From a heuristic viewpoint, integral regularity would be set by the behavior of $F$ at the ends of that space. Moreover, we explore a number of consequences of our findings, and develop some related results; these include a parabolic version of Escauriaza's exponent, a universal modulus of continuity for the solutions and estimates in $p-BMO$ spaces.
\medskip

\noindent \textbf{Keywords}:  Regularity in Sobolev spaces; nonlinear parabolic equations; asymptotic approximation methods; recession function.

\medskip

\noindent \textbf{MSC(2010)}: 35K55; 35B45.
\end{abstract}

\vspace{.1in}


\section{Introduction}\label{introduction}

In this paper, we prove regularity in Sobolev spaces for $L^p$-viscosity solutions of fully nonlinear parabolic equations of the form

\begin{equation}\label{eq_main}
	u_t\,-\,F(D^2u, Du,u,x,t)\,=\,f(x,t)\;\;\;\;\;\mbox{in}\;\;\;\;\;Q_1\,:=\,B_1\times(-1,0],
\end{equation}
where $F$ is $(\lambda,\Lambda)$-elliptic and $f\in L^{d+1}(Q_1)$ is a continuous function.

The regularity theory for nonlinear parabolic equations is a fundamental field of research in Mathematical Analysis; its applications and spillovers can be found across a wide range of disciplines, including Differential Geometry, Game Theory, Mathematical Physics, Probability, and many others.

The first main developments in the field follow from \cite{krysaf2}. In that paper, the authors address linear parabolic equations with measurable coefficients. They obtain a Harnack inequality and produce regularity of the solutions in H\"older spaces. If $u$ solves
\begin{equation}\label{eq_int1}
	u_t\,-\,F(D^2u)\,=\,0\;\;\;\;\;\mbox{in}\;\;\;\;\;Q_1,
\end{equation}
a linearization argument implies that both $u$ and its derivative in the direction $\xi\in\mathbb{R}$, $u_\xi$, solve an equation under the scope of \cite{krysaf2}. Therefore, $u\in\mathcal{C}^{1+\alpha,\frac{1+\alpha}{2}}_{loc}(Q_1)$, where $\alpha$ is unknown.

In \cite{krylov84}, the author assumes the operator $F$ is convex and prove that solutions to \eqref{eq_int1} are of class $\mathcal{C}^{2+\alpha,\frac{2+\alpha}{2}}_{loc}(Q_1)$. This result implies that under convexity assumptions on the problem, a theory of classical solutions is available. 

As regards pointwise estimates in terms of measure-theoretic quantities, we refer the reader to \cite{krylovabp76} and \cite{tso}. In those papers, the authors prove Aleksandrov-Bakelmann-Pucci estimates for linear second order equations of parabolic type.

In \cite{caffarelli89}, the author establishes several a priori estimates for solutions of fully nonlinear \textit{elliptic} equations, including regularity in Sobolev and H\"older spaces. To some extent, besides becoming a cornerstone of the profession, this trailblazing work also sets the program for the parabolic realm. In this context, a series of papers appearing in the early 90's - see \cite{lihewang1}, \cite{lihewang2} and \cite{lihewang3} - extends the perspective introduced in \cite{caffarelli89} to the parabolic setting. In particular, the author produces Harnack inequalities, investigates a priori H\"older regularity and examines estimates in Sobolev spaces. 

As regards a priori regularity in Sobolev spaces, the author  assumes the source term to be in $L^{d+1}(Q_1)$ and requires the oscillation with respect to the operator with frozen coefficients to be small in the $L^{d+1}(Q_1)-$sense. In addition, the author assumes $\mathcal{C}^{1,1}$-estimates are available for the operator with frozen coefficients. Under those assumptions, a priori estimates for $u_t$ and $D^2u$ in $L^p_{loc}(Q_1)$ are established - see \cite{lihewang1}.

In recent years, various further developments advanced the understanding of the regularity theory for nonlinear parabolic equations. In \cite{CKS2000}, the authors develop a theory of $L^p-$viscosity solutions for \eqref{eq_main} and prove a number of results. Some of the developments in \cite{CKS2000} are used in the present paper. See also \cite{cfks98}.

In \cite{caffstef}, the authors produce a counterexample type of result. Indeed, they consider a toy-model for \eqref{eq_main} and show that solutions may fail to be in $\mathcal{C}^{2,1}$. Sharp regularity for $p-$caloric equations is studied in \cite{teiurb1}, where the authors obtain a closed-form expression for the optimal H\"older exponent depending on the dimension $d$ and $p$. 

The solvability of parabolic fully nonlinear equations is the subject of \cite{krylov2013a}. In that paper, the authors consider a more general formulation, including examples of the Hamilton-Jacobi-Bellman equation. They prove solvability in $W^{2,1;p}$, by assuming the leading operator to be convex and positive homogeneous of degree one, with respect to the Hessian. Additional natural growth assumptions on the operators governing the problem are also required. Solvability in Lebesgue spaces, in the presence of VMO coefficients, is the subject of \cite{krylov07a} and \cite{krylov07b}.

In \cite{jvtei}, the authors examine optimal regularity for nonlinear parabolic equations in the presence of source terms in anisotropic Lebesgue spaces $L^{p,q}(Q_1)$. Moreover, they study distinct regularity regimes - depending on $d$, $p$, and $q$ - obtaining exact expressions for the associated H\"older exponents. Under slightly stronger assumptions on the governing operator, the authors also prove that solutions are in $C^{1,\llip}$. A survey of the parabolic theory, detailing foundational results, may be found in \cite{imbertsilv}.

In the present paper, we prove that solutions to \eqref{eq_main} are in $W^{2,1;p}_{loc}(Q_1)$, under fairly general, asymptotic, assumptions on the governing operator $F$. We argue by means of an approximation method. In brief, we design a path relating our problem of interest to an auxiliary one. 

In our concrete case, we use to the notion of recession function, formally defined as $F^*(M):=\infty^{-1}F(\infty M)$. From a heuristic viewpoint, this operator accounts for the behavior of $F$ at the ends of $\mathcal{S}(d)$, encoding an asymptotic analysis of the problem. The idea of recession function - borrowed from the realm of convex analysis - appears in the context of regularity theory in \cite{silvtei} and \cite{pimtei}. We detail the notion of recession function in Section \ref{sec_gtaic}.

We also make use of an oscillation measure; for fixed $(x_0,t_0)\in Q_1$ define 
\[
	\beta_{F^*}(x_0,t_0,x,t)\,:=\,\sup_{M\in\mathcal{S}(d)}\frac{|F^*(M,0,0,x,t)\,-\,F^*(M,0,0,x_0,t_0)|}{1\,+\,\|M\|}.
\]
This quantity was introduced in \cite{caffarelli89}. Our main theorem reads as follows:

\begin{teo}\label{maintheorem}
Let $u$ be a normalized viscosity solution to \eqref{eq_main} and assume $f\in L^p(Q_1)$, for $p>d+1$. Suppose that $F^*$ has $\mathcal{C}^{2+\alpha,\frac{2+\alpha}{2}}-$estimates. Suppose further that $\beta_{F^*}$ satisfies
\[
	\left(\frac{1}{|Q_r|}\int_{Q_r}\left|\beta_{F*}(x_0,t_0,x,t)\right|^{d+1}dxdt\right)^\frac{1}{d+1}\,\leq\,Cr^\alpha,
\]
for every $(x_0,t_0)\in Q_1$, for some $C>0$ and $\alpha\in(0,1)$. Then, $u_t$ and $D^2u$ are in $L^p(Q_{1/2})$ and satisfy
\[
	\left\|u_t\right\|_{L^{p}(Q_{1/2})}\,+\,\|D^2u\|_{L^{p}(Q_{1/2})}\,\leq\,C\left(\left\|u\right\|_{L^\infty(Q_1)}\,+\,\left\|f\right	\|_{L^{p}(Q_1)}\right),
\]
for some $C>0$.
\end{teo}

The proof of Theorem \ref{maintheorem} proceeds in two main steps. First, we investigate equations governed by operators without explicit dependence on the function and on the gradient. This step amounts to establish the following proposition:

\begin{Proposition}\label{mainprop}
Let $u$ be a normalized viscosity solution to 
\begin{equation}\label{eq_model}
	u_t\,-\,F(D^2u,x,t)\,=\,f(x,t)\;\;\;\;\;\mbox{in}\;\;\;\;\;Q_1\,:=\,B_1\times(-1,0],
\end{equation}
Assume $f\in \mathcal{C}(Q_1)\cap L^p(Q_1)$, for $p>d+1$. Further, assume that $F^*$ has $\mathcal{C}^{2+\alpha,\frac{2+\alpha}{2}}-$estimates and $\beta_{F^*}(x_0,t_0,x,t)$ satisfies 
\[
	\left\|\beta_{F^*}(x_0,t_0,\cdot,\cdot)\right\|_{L^p(Q_1)}\,\ll\,1/2,
\]
for every $(x_0,t_0)\in Q_1$, fixed.
Then, $u_t$ and $D^2u$ are in $L^p(Q_{1/2})$ and satisfy

\[
	\left\|u_t\right\|_{L^{p}(Q_{1/2})}\,+\,\|D^2u\|_{L^{p}(Q_{1/2})}\,\leq\,C\left(\left\|u\right\|_{L^\infty(Q_1)}\,+\,\left\|f\right	\|_{L^{p}(Q_1)}\right),
\]
for some $C>0$.
\end{Proposition}

The proof of Proposition \ref{mainprop} combines two sets of techniques. First, we use standard measure-theoretical results; those allow us to examine the quantities $\| u_t \|_{L^p(Q_1)}$ and $\| D^2u\|_{L^p(Q_1)}$ in terms of the measure of certain subsets of $Q_1$. Then, asymptotic approximation methods yield appropriate, improved, decay rates for the measure of those sets. 

The second step of the proof of Theorem \ref{maintheorem} involves properties of $L^p$-viscosity solutions of parabolic nonlinear equations. Together with standard regularity results, those properties build upon Proposition \ref{mainprop} to complete the proof of Theorem \ref{maintheorem}.


Our findings produce a number of consequences to the general theory of nonlinear parabolic PDEs. For example, we obtain a priori regularity for solutions to \eqref{eq_model} in $p-$BMO spaces. To the best of our knowledge, a priori regularity in BMO spaces for parabolic fully nonlinear equations had not yet been considered in the literature. Such a class of results is relevant for it bridges the gap between the $W^{2,1;p}_{loc}(Q_1)$ and $W^{2,1;\infty}_{loc}(Q_1)$ spaces. We compare this gain of regularity with the improvement represented by estimates in $\mathcal{C}^{\llip}$ \textit{vis-a-vis} estimates in $\mathcal{C}^{\alpha,\frac{\alpha}{2}}$, for every $\alpha\in(0,1)$. Our developments in this direction relate, to some extent, to previous results obtained in the elliptic setting; we mention, for example, \cite{caffhuang}. 

When studying Sobolev estimates in the elliptic setting, it is standard to assume the existence of $\epsilon>0$, universal, so that $(d-\epsilon)$-integrability of the source term would suffice for the development of the theory. This number is known in the literature as \textit{Escauriaza's exponent}. A natural question refers to the parabolic analog of such a constant. Although such a result is expected to hold true - see \cite[Remark I]{escauriaza93} - no proof had yet been produced. We recall a Harnack inequality and establish the existence of the parabolic Escauriaza's exponent. 

As a spillover of this Harnack inequality, we obtain a universal modulus of continuity for the solutions of \eqref{eq_model}; see \cite{teixeirauniversal}, c.f. \cite{jvtei}. In particular, we produce a sharp universal exponent, given by
\[
	\alpha^*\,=\,\alpha^*(d,\epsilon)\,=\,\frac{d-2\epsilon}{d+1-\epsilon}.
\]

The remainder of this paper is organized as follows: Section \ref{sec_maop} presents some notation, details the main assumptions under which we work and recall preliminary results of the theory. An asymptotic approximation method is the subject of Section \ref{sec_gtaic}, whereas the proof of Theorem \ref{maintheorem} is presented in Section \ref{sec_sobolev}. In Section \ref{sec_escauriaza}, we obtain an improved Harnack inequality and examine the parabolic analog of Escauriaza's exponent; in Section \ref{sec_um}, an approximation result builds upon this improved Harnack inequality to produce a universal modulus of continuity for the solutions. Closing the paper, Section \ref{sec_bmo} contains a study of regularity in $p-BMO$ spaces.

\paragraph{Acknowledgments} For valuable comments and suggestions on the material in this paper, the authors are grateful to B. Sirakov, A. \'{S}wi{\polhk{e}}ch, E. Teixeira and an anonymous referee. 

\bigskip

\noindent R. Castillo is funded by CAPES-Brazil; E. Pimentel was partially supported by FAPESP (Grant \# 2015/13011-6) and PUC-Rio baseline funds.

\section{Notation, key assumptions and preliminary results}\label{sec_maop}

In this section, we present some notation and detail the main assumptions under which we work. We also collect some preliminary results, for future reference.

\subsection{Elementary notation}\label{subsec_notation}

We define the \textit{parabolic domain} $Q_\rho$ as follows:
\[
	Q_\rho\,:=\,\left\lbrace (x,t)\in\Rr^{d+1}\,:\,x\in B_\rho,\,t\in(-\rho^2,0)\right\rbrace\,=\,B_\rho\,\times\,(-\rho^2,0].
\]
The parabolic boundary of $Q_\rho$ is denoted by $\partial_pQ_{\rho}$ and given by
\[
	\partial_pQ_\rho\,:=\,B_\rho\times \lbrace t=-\rho^2 \rbrace\,\cup\,\partial B_\rho\times(-\rho^2,0].
\]
Our main result respects norms of $u$ in the Sobolev space $W^{2,1;p}_{loc}(Q_1)$; we set
\[
	\left\|u\right\|_{W^{2,1;p}(Q_{1/2})}\,:=\,\left(\left\|u_t\right\|_{L^{p}(Q_{1/2})}\,+\,\|D^2u\|_{L^{p}(Q_{1/2})}\right)^\frac{1}{p};
\]
we say $u\in W^{2,1;p}(Q_{1/2})$ if there is a constant $C>0$ so that
\[
	\left\|u\right\|_{W^{2,1;p}(Q_{1/2})}\,\leq\,C,
\]
for $p>1$.

Because our developments touch the H\"older regularity theory, we continue by detailing the parabolic norms in those function spaces. Let $X=(x,t)$ and $Y=(y,s)$ be points in $Q_1$; we define the parabolic distance between $X$ and $Y$ as 
\[
	d(X,Y)\,:=\,|x-y|\,+\,|t-s|^{1/2}.
\]
We say that $u\in\mathcal{C}^{\alpha,\frac{\alpha}{2}}(Q_1)$ if there exists a constant $C>0$ so that
\[
	|u(x,t)\,-\,u(y,s)|\,\leq\,C\left(|x-y|^\alpha\,+\,|t-s|^\frac{\alpha}{2}\right),
\]
for every $(x,t)$ and $(y,s)$ in $Q_1$.

The parabolic cube of side $\rho$, denoted by $K_\rho$, is given by
\[
	K_\rho\,:=\,\left[-\rho,\rho\right]^d\times[-\rho^2,0].
\]
Given $K_\rho$, we obtain dyadic cubes of the $i-$th generation by properly bisecting the sides of the predecessor $K_{\rho}$; those are denoted by $K_{\rho/2^i}$.

Because we work under the framework of viscosity solutions, we briefly recall the definition of the class $S(\lambda,\Lambda,f)$. The Pucci's extremal operators $\mathcal{M}^\pm$ are given by
\[
\mathcal{M}^-(M,\lambda,\Lambda)\,=\,\mathcal{M}_{\lambda,\Lambda}^-(M)\,:=\,\lambda\sum_{e_i>0}e_i\,+\,\Lambda\sum_{e_i<0}e_i
\]
and
\[
\mathcal{M}^+(M,\lambda,\Lambda)\,=\,\mathcal{M}_{\lambda,\Lambda}^+(M)\,:=\,\Lambda\sum_{e_i>0}e_i\,+\,\lambda\sum_{e_i<0}e_i,
\]
where $e_i$ are the eigenvalues of the matrix $M$.

\begin{Definition}[The class of viscosity solutions]\label{def_visc} Let $f$ be a continuous function in a parabolic domain $Q$ and consider $0<\lambda\leq\Lambda$. We denote by $\overline{S}(\lambda,\Lambda,f)$ the space of continuous functions $u$ so that
\[
	u_t\,-\,\mathcal{M}^+(D^2u,\lambda,\Lambda)\,\geq\,f(x,t)\;\;\;\;\;\mbox{in}\;\;\;\;\;Q,
\]
in the viscosity sense. Similarly, $\underline{S}(\lambda,\Lambda,f)$ is the space of continuous functions $u$ so that
\[
	u_t\,-\,\mathcal{M}^-(D^2u,\lambda,\Lambda)\,\leq\,f(x,t)\;\;\;\;\;\mbox{in}\;\;\;\;\;Q.
\]
\end{Definition}

As in \cite{caffarelli89}, our argument relies on the refinement of a decay rate for the measure of certain sets. Let $M>0$; the paraboloid of opening $M$ is denoted by $P_M$ and defined as
\[
	P(x,t)\,=\,L(x,t)\,\pm\,M\left(\left|x\right|^2+|t|\right),
\]
where $L:Q_1\to\mathbb{R}$ is an affine function. Given an open subset $Q\subset Q_1$ and $u:Q\to\mathbb{R}$, we define
\[
	\underline{G}_M(u,Q)\,:=\,\left\lbrace (x_0,t_0)\in Q\,:\, \exists\; P_M\;\mbox{s.t.}\; P_M(x,t)\,\leq \,u(x,t)\;\mbox{and}\;P_M(x_0,t_0)\,=\,u(x_0,t_0)\right\rbrace,
\]
\[
	\overline{G}_M(u,Q)\,:=\,\left\lbrace (x_0,t_0)\in Q\,:\, \exists\; P_M\;\mbox{s.t.}\; P_M(x,t)\,\geq \,u(x,t)\;\mbox{and}\;P_M(x_0,t_0)\,=\,u(x_0,t_0)\right\rbrace
\]
and
\[
	G_M(u,Q)\,:=\,\underline{G}_M(u,Q)\,\cap\,\overline{G}_M(u,Q).
\]

The set $G_M(u,Q)$ comprises the points $(x,t)\in Q$ that can be touched by paraboloids of opening $M$ from above and from below. In a similar way, we have
\[
	\underline{A}_M(u,Q)\,:=\,Q\setminus\underline{G}_M(u,Q)\;\;\;\;\;\;\;\;\;\;\overline{A}_M(u,Q)\,:=\,Q\setminus\overline{G}_M(u,Q)
\]
and
\[
	A_M(u,Q)\,=\,\underline{A}_M(u,Q)\cup\overline{A}_M(u,Q).
\]

A priori Sobolev regularity for solutions of \eqref{eq_main} is studied using refined decay rates for the measure of the sets $A_M$, in terms of $M$. I.e., Hessian's integrability, as well as integrability of $u_t$, depend on the smallness of the sets of points that cannot be touched by a paraboloid of arbitrarily large opening. See \cite{caffarelli89} and \cite{ccbook}.

\subsection{Main assumptions}\label{subsec_ma}

We continue by detailing the hypotheses under which we work in the forthcoming sections.

\begin{Assumption}[Ellipticity]\label{Felliptic}There are constants $0\,<\,\lambda\,\leq\,\Lambda$ and $\gamma>0$ and a modulus of continuity $\omega:\mathbb{R}^+\to\mathbb{R}^+$ such that
\begin{align*}
	&\mathcal{M}_{\lambda,\Lambda}^-(M\,-\,N)\,-\,\gamma|p\,-\,q|\,-\,\omega(|r\,-\,s|)\\
	&\leq\, F(M,p,r,x,t)\,-\,F(N,q,s,x,t) \\
	&\leq\, \mathcal{M}_{\lambda,\Lambda}^+(M-N)+\gamma|p-q|\,+\,\omega(|r\,-\,s|)
\end{align*}
for every $M,\,N\in\mathcal{S}(d)$, $p,\,q\in\mathbb{R}^d$ and $r,\,s\in\mathbb{R}$.
\end{Assumption}

The former assumption concerns the uniform ellipticity of the operator. Among other things, A\ref{Felliptic} ensures that $F$ is $k-\lipop$ with respect to the Hessian, where $k:=\max\lbrace\lambda,\,\Lambda\rbrace$.

\begin{Assumption}[Regularity of the source]\label{fcontinuous}
We assume $f\in \mathcal{C}(B_1)\cap L^p(Q_1)$, for $p>d+1$.
\end{Assumption}


The requirement $p\,>\,d\,+\,1$ can be weakened; in fact, as in the elliptic case, one can prove that there exists $\epsilon>0$, such that our results hold under the condition $p\,>\,d\,+\,1\,-\,\epsilon$. 

The pivotal notion behind the asymptotic approximation method is to connect a problem of interest to another one, for which a well-established theory is available. In our concrete case, the limiting profile is assumed to have $\mathcal{C}^{2+\alpha,\frac{2+\alpha}{2}}$-estimates.


\begin{Assumption}[$\mathcal{C}^{2+\beta,\frac{2+\beta}{2}}$-estimates; case I]\label{c11rec}
We assume that solutions to
\[
	v_t\,-\,F^*(D^2v,0,0,x,t)\,=\,0\;\;\;\;\;\mbox{in}\;\;\;\;\;Q_1
\]
are such that  $v\in \mathcal{C}^{2+\beta,\frac{2+\beta}{2}}_{loc}(Q_1)$ and, in addition,
\begin{equation*}\label{c11def}
	\left\|v\right\|_{\mathcal{C}^{2+\beta,\frac{2+\beta}{2}}(Q_{1/2})}\,\leq\, C,
\end{equation*}
for some constant $C>0$.
\end{Assumption}

To include operators depending on $x$ and $t$ under the scope of our results, we need to impose an additional smallness condition. This is the content of the next assumption.

\begin{Assumption}[Oscillation at the recession level]\label{assump_osc}
Consider the oscillation measure
\[
	\beta_{F^*}(x_0,t_0,x,t)\,:=\,\sup_{M\in\mathcal{S}(d)}\frac{|F^*(M,0,0,x,t)\,-\,F^*(M,0,0,x_0,t_0)|}{1\,+\,\|M\|}.
\]
We assume
\[
	\left(\frac{1}{|Q_r|}\int_{Q_r}\left|\beta_{F*}(x_0,t_0,x,t)\right|^{d+1}dxdt\right)^\frac{1}{d+1}\,\leq\,Cr^\alpha,
\]
for some $\alpha\in(0,1)$ and some constant $C>0$.
\end{Assumption}

Assumption A\ref{assump_osc} builds upon former results (see \cite[Theorem 1.1]{lihewang2}) to yield appropriate regularity for the approximating function. Finally, to examine regularity in $p-$BMO spaces, we use a slightly stronger assumption on the limiting profile $F^*$, namely:

\begin{Assumption}[$\mathcal{C}^{2+\beta,\frac{2+\beta}{2}}$-estimates; case II]\label{c2arec}
There exists a constant $L\gg 1$ such that $F\equiv F^*$ in $\mathcal{S}(d)\setminus B_L$. Also, we assume that solutions to
\[
	v_t\,-\,F^*(D^2v,x,t)\,=\,0\;\;\;\;\;\mbox{in}\;\;\;\;\;Q_1
\]
are such that  $v\in \mathcal{C}^{2+\beta,\frac{2+\beta}{2}}_{loc}(Q_1)$, with
\begin{equation}\label{c2adef}
	\left\|v\right\|_{\mathcal{C}^{2+\beta,\frac{2+\beta}{2}}(Q_1)}\,\leq\,C,
\end{equation}
for some $C>0$.
\end{Assumption}

We use some standard results on fully nonlinear parabolic equations. For the sake of completeness, we recall those in the next section.

\subsection{Preliminary results}\label{subsec_elemres}

We start with a lemma on the stability of viscosity solutions.


\begin{Lemma}[Stability Lemma]\label{stabilitylemma}
Let $F_m:\mathcal{S}(d)\times\mathbb{R}^d\times \mathbb{R}\times Q_1\to\mathbb{R}$ satisfy A\ref{Felliptic} and $f_m:Q_1\to\mathbb{R}$ satisfy A\ref{fcontinuous}, for $m\in\mathbb{N}$. Suppose $F_m\to F_\infty$ uniformly in compact sets and $f_m\to f_\infty$ in the $L^p$-sense. If there is $(u_m)_{m\in\mathbb{N}}$ so that
\[
	\left(u_m\right)_t\,-\,F_m(D^2u_m, Du_m,u_m,x,t)\,=\,f_m(x,t)\;\;\;\;\;\mbox{in}\;\;\;\;\;Q_1,
\]
and
\[
	u_m\to u_\infty,
\]
uniformly in compact sets of $Q_1$, we have
\[
	\left(u_\infty\right)_t\,-\,F_\infty(D^2u_\infty,Du_\infty,u_\infty,x,t)\,=\,f_\infty(x,t)\;\;\;\;\;\mbox{in}\;\;\;\;\;Q_1.
\]
\end{Lemma}

For the proof of Lemma \ref{stabilitylemma}, we refer the reader to \cite[Theorem 6.1]{CKS2000}. Next, we recall a standard result on the existence of a suitable barrier function.

\begin{Lemma}[Barrier function]\label{lem_auxiliar1}Let $\rho\in(0, 1/(3\sqrt{d}))$. Then, there exists a function $\phi:Q_1\to \mathbb{R}$ so that $\phi\geq 1$ in $K_3$, $\phi\leq 0$ on $\partial_p Q_1$ and
\[
	\phi_t\,-\,\mathcal{M}_{\lambda,\Lambda}^-(D^2\phi)\,\leq\, 0
\]
in $Q_1\setminus K_\rho$. In addition,
\[
	\left\|\phi\right\|_{\mathcal{C}^{1,1}(\overline{Q}_1)}\,\leq\, C,
\]
where $C=C(\lambda,\Lambda,d)$.
\end{Lemma}

For the proof of Lemma \ref{lem_auxiliar1}, we refer the reader to \cite[Lemma 3.22]{lihewang1}. The existence of such a barrier function $\phi$ is critical in controlling the measure of certain sets. The first step in this direction is the study of the contact set for an auxiliary function of the form $w=u-2\phi$. We proceed by rigorously defining the contact set of a continuous function.

\begin{Definition}[Contact set]\label{def_contactset}Let $Q\subset \mathbb{R}^{d+1}$ and suppose $u\in\mathcal{C}(Q)$. The convex envelope of $u$ is given by
\[
	\Gamma_u(x,t)\,:=\,\sup_{L}\left\lbrace L(x,t)\,:\, L(x,t)\leq u(x,t)\right\rbrace.
\]
The contact set of $u$ is
\[
	\left\lbrace (x,t)\in Q\;:\; u(x,t)\,=\,\Gamma_u(x,t)\right\rbrace\,=\,\left\lbrace u\,=\,\Gamma_u\right\rbrace.
\]
\end{Definition}

Given $\rho_0\,\geq \,0$ we are interested in a universal lower bound for the measure of
\[
	\left\lbrace u\,=\,\Gamma_u\right\rbrace\cap Q_{\rho_0};
\]
this is the content of the next lemma:

\begin{Lemma}[Measure of the contact set]\label{lem_auxiliar2}Let $u\in S(f)$ satisfy $\left\|u\right\|_{L^\infty(Q_1)}\leq1$ and set $w:=u-2\phi$. Then, there exists $\alpha=\alpha(\lambda,\Lambda,d)$ such that
\[
	\left|\left\lbrace w\,=\,\Gamma_w\right\rbrace\,\cap\,Q_{\rho_0}\right|\,\geq 1-\alpha,
\]
for every $\rho_0\ll 1$.
\end{Lemma}
We refer the reader to \cite[Lemma 4.1]{lihewang1} for a proof of this result. In the sequel, we put forward an asymptotic approximation method and present the machinery through which it operates in this paper. 

\section{An approximation method}\label{sec_gtaic}

In this section, we detail an approximation method. At the core of our techniques, is the notion of recession function. See, for example, \cite{silvtei} and \cite{pimtei}. This set of methods is central in the proof of Proposition \ref{mainprop}. Therefore, we consider here operators of the form $F=F(M,x,t)$.

Let $F$ be a $(\lambda,\Lambda)$-elliptic operator and denote by $F_\mu:\mathcal{S}(d)\times Q_1\to\mathbb{R}$ the following object:
\[
	F_\mu(M,x,t)\,:=\,\mu F(\mu^{-1}M,x,t),
\]
for $\mu>0$.

The recession function of $F$ is denoted by $F^*$ and given by 
\[
	F^*(M,x,t)\,:=\,\lim_{\mu\to0}F_\mu(M,x,t).
\]

The operator $F^*$ accounts for the behavior of $F$ at the ends of $\mathcal{S}(d)$. Its definition also resembles the notion of a derivative at the infinity of the space. Next, we detail a few facts related to the recession function.

Because the definition of recession function involves the operation of taking limits, it is key that we ensure the convergence - in some appropriate sense - of $F_\mu$ to $F^*$. Since $F_\mu$ is $(\lambda,\Lambda)-$elliptic, for every $\mu\in(0,1)$, we have $F_{\mu}\in K-\lipop$. Hence, compactness implies that $F_\mu$ converges, through a subsequence if necessary to a recession profile $F^*$. The next proposition was established in \cite{silvtei} and plays an instrumental role in our analysis.

\begin{Proposition}[Local uniform convergence]\label{prop_luc}
Let $F$ be a $(\lambda,\Lambda)-$elliptic operator. Then, for every $\epsilon>0$ there exists $\delta>0$ so that
\[
\left\| F_\mu(M,x,t)\,-\,F^*(M,x,t)\right\|\,\leq\,\epsilon(1\,+\,\left\|M\right\|),
\]
for every $M\in\mathcal{S}(d)$, provided $\mu\,\leq\,\delta$.
\end{Proposition}

Proposition \ref{prop_luc} assures $F_{\mu}$ converges to $F^*$ uniformly in compacts of $\mathcal{S}(d)$. We notice this is precisely one of the requirements of the Stability Lemma (see Lemma \ref{stabilitylemma}). We observe that instead of imposing $F\equiv F^*$ outside of a large ball of $\mathcal{S}(d)$ in A\ref{c2arec}, we could have assumed  $F_\mu\to F^*$ globally uniformly. We believe A\ref{c2arec} simplifies the presentation. 

Here, solutions to the equation governed by $F^*$ have $W^{2,1;\infty}_{loc}(Q_1)$ a priori estimates; this is the content of A\ref{c11rec}. For small values of $\mu$, the path designed by $F_\mu$ would incorporate this property, at least partially - say, $W^{2,1;p}_{loc}$ estimates. Finally, we expect to transport this regularity back to the case $\mu=1$, i.e., to the solutions of the equation driven by $F$.

The appropriate way to formalize this intuition is by an approximation lemma.

\begin{Proposition}[Approximation Lemma]\label{lem_approx}Let $u$ be a normalized viscosity solution of
\begin{equation}
	u_t - F_{\mu}(D^{2}u,x,t)= f(x,t)\;\;\;\;\;\mbox{in}\;\;\;\;\;Q_1,
\end{equation}
and assume that A\ref{Felliptic}-A\ref{assump_osc} are satisfied. Given $ \delta > 0 $, there exists $ \epsilon\, >\, 0 $, such that, if
\begin{equation*}
	\mu\;+\;\|f\|_{L^{p}(Q_1)} \leq \epsilon,
\end{equation*}
there exists $h \in C^{2+\alpha,\frac{2+\alpha}{2}}(Q_{3/4}) $, solution to
\begin{equation}\label{eqn_eqh}
	\begin{cases}
		h_{t} - F^*(D^{2}h,x,t) \, = \,0  &\;\;\;\;\mbox{ in }\;\;\;\;Q_{3/4},  \\
		h\,=\, u &\;\;\;\;\mbox{on}\;\;\;\;\partial Q_{3/4}
	\end{cases}
\end{equation}
satisfying
\begin{equation*}
	\left\|u - h \right\|_{L^{\infty}(Q_{1/2})}\,\leq\,\delta.
\end{equation*}
\end{Proposition}
\begin{proof}

We prove this proposition by way of contradiction; suppose its statement is false. Then, there exists a number $\delta_0>0$ so that, for every solution $h$ of \eqref{eqn_eqh} we have
\[
\left\|u - h \right\|_{L^{\infty}(Q_{1/2})}\,\geq\,\delta_0,
\]
irrespective of how small $\epsilon>0$ is taken.

Let $\mu_n \sim 1/n$ and consider the sequence of operators $F_{\mu_n}$; moreover, let $(f_n)_{n\in\mathbb{N}}$ be such that
\[
	\left\| f_n\right\|_{L^{p}(Q_{3/4})}\;\to\;0,
\]
as $n\to\infty$.
Let $u_n$, solve
\[
	\left(u_n\right)_t\,-\,F_{\mu_n}(D^2u_n,x,t)\,=\,f_n\;\;\;\;\;\mbox{in}\;\;\;Q_1,
\]
and notice that $(u_n)_{n\in\mathbb{N}}$ is uniformly bounded in $\mathcal{C}^{\alpha,\frac{\alpha}{2}}$, for some $\alpha\in(0,1)$, independent of $n$. Therefore,
\[
	u_n\;\to\;u_\infty\;\;\;\;\;\mbox{in}\;\;\;\;\;\mathcal{C}^{\alpha,\frac{\alpha}{2}},
\]
through a subsequence, if necessary. The Stability Lemma (Lemma \ref{stabilitylemma}) implies that
\[
	\left(u_\infty\right)_t\,-\,F^*(D^2u_\infty,x,t)\,=\,0.
\]
We notice that A\ref{assump_osc} implies $u_\infty\in\mathcal{C}^{2+\alpha,\frac{2+\alpha}{2}}_{loc}(B_1)$; see \cite[Theorem 1.1]{lihewang2}. By choosing $h\equiv u_\infty$, we obtain a contradiction and complete the proof.
\end{proof}
Proposition \ref{lem_approx} is key in our arguments; it builds upon a measure-theoretical analysis to yield information about the integrability of solutions to \eqref{eq_main}. This analysis is the subject of the forthcoming section.

\section{A priori Sobolev regularity}\label{sec_sobolev}

In the present section, we detail the proof of Theorem \ref{maintheorem}. As previously discussed, our argument evolves along two main steps. First, we consider operators depending only on the Hessian, the space variable $x$ and time $t$.

\subsection{Proof of the Proposition \ref{mainprop}}\label{subsec_hessianonly}

Next, we derive lower, universal, integrability for $u_t$ and $D^2u$, from the ellipticity of $F$ and the integrability of the source term. Second, the Approximation Lemma (Lemma \ref{lem_approx}) connects our problem of interest with the homogeneous PDE governed by the recession operator. When combined, these steps produce improved Sobolev regularity for solutions of \eqref{eq_main}, concluding the proof. 

Throughout this section, $Q$ stands for a parabolic domain containing $Q_{8\sqrt{d}}$. We start by presenting a first decay rate for the measure of the sets $A_M$.

\begin{Proposition}[A priori regularity in $W^{2,1;\delta}_{loc}(Q_1)$]\label{prop_w2delta}
Assume that A\ref{Felliptic}-A\ref{fcontinuous} hold and let $u$ be a normalized viscosity solution to \eqref{eq_model}.
Then, there exist a universal constant $C>0$ and $\delta>0$, unknown, such that
\[
	|A_t(u,Q)\cap K_1|\,\leq\,Ct^{-\delta} 
\]
\end{Proposition}

Proposition \ref{prop_w2delta} is the parabolic analog of the celebrated $W^{2,\delta}$ estimates, well-known in the elliptic case (c.f. \cite{fanghuaw2delta}). This proposition appeared for the first time in \cite{lihewang1}.

We observe that such a priori estimate is independent of further assumptions on the operator $F$, and follows merely from uniform ellipticity and the integrability of the source term. To obtain a finer control on the integrability of solutions, we use the approximation method. By imposing a condition on the behavior of $F$ at the ends of $\mathcal{S}(d)$, we can refine the decay rate in Proposition \ref{prop_w2delta}. Next, we produce a first lower bound for the measure of $G_M\cap K_1$, for some $M>1$ universal.

\begin{Proposition}\label{prop75}
Assume A\ref{Felliptic}-A\ref{fcontinuous} are in force. Let $u\in \overline{S}(f)$ in $Q_{6\sqrt{d}}$ with $\left\|u\right\|_{L^\infty(Q)}\leq 1$. Then, there exist universal constants $\alpha\in(0,1)$, $M>1$ and $0<\delta\ll 1$, such that  $\left\|f\right\|_{L^{d+1}(Q_{6\sqrt{d}})}\leq \delta$ implies
\[
	\left|\underline{G}_M(u,Q)\cap K_1\right|\,\geq\, 1-\alpha.
\]
\end{Proposition}
\begin{proof}
The result follows along the same lines as in the proof of Lemma 7.5 in \cite{ccbook} or in the remark after Lemma 3.22 in \cite{lihewang1}, provided the necessary modifications are taken into account.
\end{proof}

From the heuristic viewpoint, A\ref{c11rec} implies a change of regime for \eqref{eq_main}; whenever $u_t$ or $D^2u$ grow \textit{too much}, the PDE is governed by the recession operator $F^*$, for which $\mathcal{C}^{2+\alpha,\frac{2+\alpha}{2}}$ estimates are available. Intuitively, it sets an upper bound for those quantities and the original operator resumes driving the problem.

When gathered with Proposition \ref{prop_w2delta}, this interplay produces faster decay rates for the measure of $A_M\cap K_1$, ultimately establishing Theorem \ref{maintheorem}. This description accounts for the asymptotic operation of the recession strategy. The next proposition translates such operation into a primary level of improved decay rates.
\begin{Proposition}\label{prop_710}
	Assume A\ref{Felliptic}-A\ref{assump_osc} are in force. Let $u$ be a normalized viscosity solution to
	\[
		u_t\,-\,F_\mu(D^2u,x,t)\,=\,f(x,t)\;\;\;\;\;\mbox{in}\;\;\;\;\;Q_{8\sqrt{d}},
	\]
so that
	\[
		-\left|x\right|^2\,-\,\left|t\right|\,\leq\,u(x,t)\,\leq\,\left|x\right|^2\,+\,\left|t\right|,
	\]
	in $Q\setminus Q_{8\sqrt{d}}$. Assume further that
	\[
		\mu\,+\,\left\|f\right\|_{L^{d+1}(Q_{8\sqrt{d}})}\,\leq\, \epsilon.
	\]
Then, there exists $M>1$ such that
	\[
	\left|G_M(u,Q)\cap K_1\right|\,\geq\,1-\epsilon_0,
	\]
	for $\epsilon_0\in(0,1)$.
\end{Proposition}
\begin{proof}
Consider the function $h$, $\epsilon$-close to $u$, given by Proposition \ref{lem_approx}; extend $h$ continuously to $Q$ in such a way that
\[
	h\,=\,u\;\;\;\;\;\mbox{in}\;\;\;\;\;Q\setminus Q_{7\sqrt{d}}
\]
and
\[
	\left\|u\,-\,h\right\|_{L^\infty(Q)}\,=\,\left\|u-h\right\|_{L^\infty(Q_{6\sqrt{d}})}.
\]
In addition, the maximum principle implies
\[
	\left\|u\right\|_{L^\infty(Q_{6\sqrt{d}})}\,=\,\left\|h\right\|_{L^\infty(Q_{6\sqrt{d}})};
\]
hence, $\left\|u-h\right\|_{L^\infty(Q)}\leq 2$ and
\[
		-2-\left|x\right|^2\,-\,\left|t\right|\,\leq\,h(x,t)\,2+\leq\,\left|x\right|^2\,+\,\left|t\right|\;\;\;\;\;\mbox{in}\;\;\;\;\;Q\setminus Q_{6\sqrt{d}}.
\]
Therefore, there exists $N>1$ for which $Q_1\subset G_N(h,Q)$.

Next, set
\[
	w\;:=\;\frac{\delta}{2C\epsilon}(u\,-\,h).
\]
Because $w$ satisfies the assumptions of Proposition \ref{prop_w2delta}, it follows that
\[
	\left|A_t(w,Q)\cap K_1\right|\,\leq\, Ct^{-\sigma}\;\;\;\;\;\forall t>0,
\]
and
\[
\left|A_s(u-h,Q)\cap K_1\right|\,\leq\, C\epsilon^{-\sigma}s^{-\sigma}\;\;\;\;\;\forall s>0.
\]
This, in turn, yields
\[
\left|G_N(u-h,Q)\cap K_1\right|\,\geq\,1\,-\, C\epsilon^{-\sigma}s.
\]
By choosing $\epsilon\ll 1$ appropriately, and setting $M\equiv 2N$, the proof is concluded.
\end{proof}

An application of Proposition \ref{prop_710} produces valuable information on the measure of $G_M\cap K_1$, provided $G_1(u,Q)\cap K_3$ is not empty. The next proposition yields the first step of an iteration scheme appearing later in this section.

\begin{Proposition}\label{prop711}
	 Assume A\ref{Felliptic}-A\ref{assump_osc} are in force and suppose $u$ is a normalized viscosity solution of
	 \[
		 u_t\,-\,F_\mu(D^2u,x,t)\,=\,f(x,t)\;\;\;\;\;\mbox{in}\;\;\;\;\;Q_{8\sqrt{d}}.
	 \]
	 Assume further that
	 \[
	 \mu\,+\,\left\|f\right\|_{L^{d+1}(Q_{8\sqrt{d}})}\leq \epsilon.
	 \]
	 Finally, suppose $G_1\cap K_3\neq \varnothing$. Then,
	 \[
	 \left|G_M(u,Q)\cap K_1\right|\,\geq\,1\,-\,\epsilon_0,
	 \]
	 where $M>1$ and $\epsilon>0$ are taken as in Proposition \ref{prop_710}.
\end{Proposition}
\begin{proof}
	We argue by means of an auxiliary function. First, let $(x_1,t_1)\in G_1(u,Q)\cap K_3$; notice that
	\[
		-\frac{|x-x_1|^2+|t-t_1|}{2}\,\leq \,u(x,t)-L(x,t)\,\leq\,\frac{|x-x_1|^2+|t-t_1|}{2},
	\]
	where $L$ is an affine function. We define
	\[
		v\,:=\,\frac{1}{C}(u\,-\,L),
	\]
	where $C$ is chosen to ensure $\left\|v\right\|_{L^\infty(Q_{8\sqrt{d}})}\leq 1$, and
	\[
		-|x|^2\,-\,|t|\,\leq \,v(x,t)\,\leq\,|x|^2+|t| \;\;\;\;\;\mbox{in}\;\;\;\;\;Q\setminus Q_{6\sqrt{d}}.
	\]
	Moreover, $v$ solves
	\[
		\frac{1}{C}v_t\,-\,\frac{1}{C}F_\mu(CD^2v,x,t)\,=\,\frac{1}{C}f(x,t)\;\;\;\;\;\mbox{in}\;\;\;\;\;Q_{8\sqrt{d}}.
	\]
	Therefore, an application of Proposition \ref{prop_710} yields
	\[
		\left|G_M(v,Q)\cap K_1\right|\,\geq\,1\,-\,\epsilon_0,
	\]
	i.e.,
	\[
		\left|G_{CM}(u,Q)\cap K_1\right|\,\geq\,1\,-\,\epsilon_0,
	\]
	and the proposition is established.
\end{proof}

As mentioned earlier, Proposition \ref{prop711} fits into our argument as the first step of an iteration scheme that substantially improves Proposition \ref{prop_w2delta}. In this context, the former is matched by a measure-theoretical result in the spirit of Calder\'on-Zygmund decomposition, known as stacked covering lemma.

\begin{Lemma}[Stacked covering lemma]\label{lem_stacked}Fix $m\in\mathbb{N}$ and consider $A,\,B\subset Q$. Assume that:
	\begin{enumerate}
		\item there exists $\delta\in(0,1)$ so that
		\[
			\left|A\right|\,\leq\,\delta\left|Q\right|;
		\]	
		\item for any dyadic cube $K\subset Q$ so that
		\[
			\left|K\cap A\right|\,>\,\delta\left| K\right|,
		\]
		we have
		\[
			\overline{K}^{(m)}\subset B.
		\]
	\end{enumerate}	
	Then,
	\[
	\left|A\right|\,\leq\,\frac{\delta(m+1)}{m}\left|B\right|.
	\]
\end{Lemma}

A proof of Lemma \ref{lem_stacked} can be found in \cite{imbertsilv}, where the authors recur to a Lebesgue's Differentiation Theorem. As mentioned in \cite{imbertsilv}, a similar rationale underlies some of the arguments presented in \cite{lihewang1}.

In what follows, Proposition \ref{prop711} builds upon the stacked covering lemma to produce finer decay rates for the sets $A_M\cap K_1$; this is the content of our next result.

\begin{Proposition}\label{prop712}
Let $u$ be a normalized viscosity solution to \eqref{eq_model} in $Q_{8\sqrt{d}}$ and consider $\epsilon_0\in(0,1)$. Assume A\ref{Felliptic}-A\ref{assump_osc} are in force. Extend $f$ by zero outside $Q_{8\sqrt{d}}$ and define
\[
	A\,:=\,A_{M^{k+1}}(u,Q_{8\sqrt{d}})\,\cap\,K_1
\]
and
\[
	B\,:=\,\left\lbrace A_{M^k}(u,Q_{8\sqrt{d}})\,\cap\,K_1 \right\rbrace\,\cup\,\left\lbrace (x,t)\,\in\, K_1\,:\,m(f^{d+1})(x,t)\,\geq\,(CM^k)^{d+1} \right\rbrace.
\]
Then, 
\[
	|A|\,\leq\,\epsilon_0|B|,
\]
where $M>1$ depends on the dimension and $C>0$ is a universal constant.
\end{Proposition}
\begin{proof}
The proof is an application of Lemma \ref{lem_stacked}. We start by noticing that
\[
	|u(x,t)|\,\leq\,1\,\leq\,|x|^2+|t|\;\;\;\;\;\mbox{in}\;\;\;\;\;Q_{8\sqrt{d}}\setminus Q_{6\sqrt{d}}.
\]
Hence, Proposition \ref{prop_710} yields
\[
	|(A_{M^{k+1}}(u,Q_{8\sqrt{d}})\cap K_1)\cap Q_1|\,\leq\,\epsilon_0\,=\,\epsilon_0|Q_1|.
\]
This verifies the first condition in that lemma.
Now, let $K\,:=\,K_{1/2^i}$ be any dyadic cube of $K_1$ so that
\[
	|K\cap A_{M^{k+1}}(u,Q_{8\sqrt{d}})|\,>\, \epsilon_0|K|;
\]
It remains to prove that, for some $m\in\mathbb{N}$, we have $\overline{K}^{(m)}\subset B$. We verify this fact using a contradiction argument. Assume 
\[
	\overline{K}^{(m)}\not\subset B;
\]
therefore, there exists $(x_1,t_1)$ so that
\begin{equation}\label{co1}
	(x_1,t_1)\,\in\,\overline{K}^{(m)}\cap G_{M^k}(u,Q_{8\sqrt{d}})
\end{equation}
and
\begin{equation}\label{co2}
	m(f^{d+1})(x_1,t_1)\,\leq\,\left(CM^k\right)^{d+1}.
\end{equation}

Define the auxiliary function $\tilde{u}$ as follows:
\[
	\tilde{u}(y,s)\,:=\,\frac{2^{2i}}{M^k}\,u\left(\frac{y}{2^i},\frac{s}{2^{2i}}\right);
\]
notice $\tilde{u}$ is a normalized viscosity solution to
\[
	\tilde{u}_t\,-\,G(D^2\tilde{u},x,t)\,=\,\tilde{f}(x,t)\;\;\;\;\mbox{in}\;\;\;\;\;Q_{8\sqrt{d}/2^i},
\]
where 
\[
	G(D^2v,x,t)\,=\,\frac{1}{M^k}F(M^k D^2v,x,t),
\]
and
\[
	\tilde{f}(x,t)\,=\,\frac{1}{M^k}f\left(\frac{x}{2^i},\frac{t}{2^{2i}}\right).
\]
Because $F^*$ has interior $\mathcal{C}^{1,1}$ estimates, so does $G^*$. Also, 
\[
	\left\|\tilde{f}\right\|_{L^{d+1}(Q_{8\sqrt{d}})}^{d+1}\,\leq\,\frac{2^{i(d+2)}}{M^{k(d+1)}}\int_{Q_{8\sqrt{d}/2^i}}\left|f(x,t)\right|^{d+1}dxdt\,\leq\,2^{i(d+2)}C^{d+1};
\]
by choosing $C$ sufficiently small in \eqref{co2}, we conclude
\[
	\left\|\tilde{f}\right\|_{L^{d+1}(Q_{8\sqrt{d}})}\,\leq\,\epsilon.
\]
In addition, \eqref{co1} implies 
\[
	G_{1}(\tilde{u},Q_{8\sqrt{d}/2^i})\cap K_3\,\neq\,\emptyset.
\]
Therefore, Proposition \ref{prop711} yields
\[
	|G_{M^{k+1}}(u,Q_{8\sqrt{d}})\cap K|\,\geq\,(1\,-\,\epsilon_0)|K|,
\]
which leads to a contradiction and concludes the proof.
\end{proof}

Proposition \ref{prop712} states that 
\[
	|A_{M^{k+1}}(u,Q_{8\sqrt{d}})\,\cap\,K_1|\leq\epsilon_0\left|\left\lbrace A_{M^k}(u,Q_{8\sqrt{d}})\,\cap\,K_1 \right\rbrace\,\cup\,\left\lbrace m(f^{d+1})(x,t)\,\geq\,(CM^k)^{d+1} \right\rbrace\right|;
\]
because $0\leq\epsilon_0\leq 1$, the former inequality implies the summability of key quantities, ultimately yielding the proof of Proposition \ref{maintheorem}.

\begin{proof}[Proof of Proposition \ref{mainprop}]
Set
\[
	\alpha_k\,:=\,\left|A_M^k(u,Q_{8\sqrt{d}})\cap K_1\right|
\]
and
\[
	\beta_k\;:=\;\left|\lbrace (x,t)\in K_1\;:\;m(f^{d+1})(x,t)\geq(CM^k)^{d+1}\rbrace\right|.
\]
The proof is complete if we manage to verify that there is a constant $C>0$ so that
\[
	\sum_{k\geq 0}M^{pk}\alpha_k\,\leq\,C.
\]
Proposition \ref{prop712} yields
\begin{equation}\label{eqfinal}
	\alpha_k\,\leq\,\epsilon_0^k\,+\,\sum_{i=0}^{k-1}\epsilon_0^{k-i}\beta_i.
\end{equation}
On the other hand, A\ref{fcontinuous} implies $f^{d+1}\in L^{\frac{p}{d+1}}(Q_1)$; hence, $m(f^{d+1})\in L^{\frac{p}{d+1}}(Q_1)$ and we have
\[
	\left\|m(f^{d+1})\right\|_{L^{\frac{p}{d+1}}(Q_1)}\,\leq\,C\left\|f\right\|^{d+1}_{L^p(Q_1)}\,\leq\,C.
\]
The last inequality implies
\begin{equation}\label{eqfinal2}
	\sum_{k\geq 0}M^{pk}\beta_k\,\leq\,C.
\end{equation}

By combining \eqref{eqfinal} and \eqref{eqfinal2}, we finally have
\begin{align*}
	\sum_{k\geq 1}M^{pk}\alpha_k\,&\leq\, \sum_{k\geq 1}(\epsilon_0M^p)^k\,+\,\sum_{k\geq 0}\sum_{i=0}^{k-1}\epsilon_0^{k-i}M^{p(k-i)}\beta_iM^{pi}\\
		&\leq \,\sum_{k\geq1}2^{-k}\,+\,\left(\sum_{i\geq 0}M^{pi}\beta_i\right)\,+\,\left(\sum_{j\geq1}2^{-j}\right)\\
		&\leq \,C.
\end{align*}

\end{proof}

\subsection{Proof of Theorem \ref{maintheorem}}\label{subsec_gencase}

Next, we present the proof of Theorem \ref{maintheorem}. In general lines, results available for $L^p$-viscosity solutions build upon Proposition \ref{mainprop} to produce the conclusion.

\begin{proof}[Proof of Theorem \ref{maintheorem}] We split the argument in two main steps.

\vspace{.1in}

\noindent {\bf Step 1} We start with a reduction procedure. That is, we prove that it suffices to verify the result for $L^p$-viscosity solutions of the model problem \eqref{eq_model}. Because of \cite[Proposition 3.2]{CKS2000}, we know that $u$ is parabolic twice differentiable a.e.; moreover, its pointwise derivatives satisfy \eqref{eq_main} a.e. in $Q_1$. In the sequel, define $g:Q_1\to\mathbb{R}$ as
\[
	g(x,t)\,:=\,F(D^2u,0,0,x,t).
\]
Assumption A\ref{Felliptic} implies
\begin{align*}
	\left|g(x,t)\right|\,&\leq\,\left|F(D^2u,0,0,x,t)\,-\,F(D^2u,Du,u,x,t)\right|\,+\,\left|f(x,t)\right|\\
			 &\leq\, \gamma|Du|\,+\,\omega(|u|)\,+\,|f(x,t)|.
\end{align*}
Therefore, former results on the regularity of continuous viscosity solutions imply $g\in L^p_{loc}(Q_1)$ -- see \cite[Theorem 7.3]{CKS2000} or \cite{lihewang1}. Set 
\[
	G(D^2u,x,t)\,:=\,F(D^2u,0,0,x,t).
\]
By using \cite[Proposition 4.1]{CKS2000}, we conclude $u$ is an $L^p$-viscosity solution to
\[
	u_t\,-\,G(D^2u,x,t)\,=\,g(x,t)\;\;\;\;\;\mbox{in}\;\;\;\;\;Q_1.
\]
Assume now that Theorem \ref{maintheorem} is available for the $L^p$-viscosity solutions of problems without dependence on the gradient. Then, we would have $u\in W^{2,1;p}_{loc}(Q_{1})$ and 
\[
	\left\|u\right\|_{W^{2,1;p}(Q_{1/2})}\,\leq\,C\left(\left\|u\right\|_{L^\infty(Q_{1})}\,+\,\left\|g\right\|_{L^p(Q_{1})}\right),
\]
establishing the result. 

\vspace{.1in}

\noindent {\bf Step 2} In the sequel, we consider the problem
\begin{equation}\label{auxlp}
	u_t\,-\,G(D^2u,x,t)\,=\,g(x,t)\;\;\;\;\;\mbox{in}\;\;\;\;\;Q_1;
\end{equation}
although $G$ is continuous with respect to $x$ and $t$, no information about the continuity of $g$ is available. Therefore, we consider two sequences of functions: $(g_j)_{j\in\mathbb{N}}\in\mathcal{C}^\infty(\overline{Q_1})\cap L^p(Q_1)$ and $(u_j)_{j\in\mathbb{N}}$. Assume $(g_j)_{j\in\mathbb{N}}$ is such that 
\[
	\left\|g_j\,-\,g\right\|_{L^p(Q_1)}\,\to\,0\;\;\;\;\;\mbox{as}\;\;\;\;\;j\to\infty.
\]
We relate those sequences through the following family of PDEs:
\[
	(u_j)_t\,-\,G(D^2u_j,x,t)\,=\,g_j(x,t)\;\;\;\;\;\mbox{in}\;\;\;\;\;Q_1.
\]
It is clear that $G^*$ satisfies A\ref{c11rec} and A\ref{assump_osc}. Then, Proposition \ref{mainprop} implies $u_j\in W^{2,1;p}_{loc}(Q_{1})$ and
\[
	\left\|u_j\right\|_{W^{2,1;p}(Q_{1/2})}\,\leq\,C\left(\left\|u_j\right\|_{L^\infty(Q_{1})}\,+\,\left\|g_j\right\|_{L^p(Q_{1})}\right).
\]
Because of \cite[Proposition 2.6]{CKS2000}, a straightforward argument yields $u_j\to \overline{u}$ in $\mathcal{C}(\overline{Q_1})$. Notice also that $u_j$ weakly converges to $\overline{u}$ in $W^{2,1;p}_{loc}(Q_1)$. Hence, 
\[
	\left\|\overline{u}\right\|_{W^{2,1;p}(Q_{1/2})}\,\leq\,C\left(\left\|\overline{u}\right\|_{L^\infty(Q_{1})}\,+\,\left\|g\right\|_{L^p(Q_{1})}\right);
\]
moreover, stability results guarantee that $\overline{u}$ is an $L^p$-viscosity solution to \eqref{auxlp}. The maximum principle \cite[Lemma 6.2]{CKS2000}, together with compatibility on the parabolic boundary, yields $\overline{u}=u$ and concludes the proof. 
\end{proof}

\begin{Remark}
Step 2 is required because we have no information on the continuity of the functions $g$. For large values of $p>d+2$, however, \cite[Theorem 7.3]{CKS2000} ensures that $Du$ is H\"older continuous. In this case, Step 1 would suffice to establish the result. 
\end{Remark}

\begin{Remark}
In \cite{winter}, the author investigates boundary regularity in Sobolev spaces for the elliptic problem. We believe the reasoning in Step 2 could be applied to prove boundary regularity in the parabolic case as well. It would remain to produce localized versions (at the boundary) of the results in Section \ref{subsec_hessianonly}.
\end{Remark}

\section{Escauriaza's parabolic exponent}\label{sec_escauriaza}

A natural question to be considered in this setting regards the celebrated Escauriaza's exponent. In \cite{escauriaza93}, the author remarks that it would be possible to obtain a constant $\epsilon=\epsilon(d,\lambda,\Lambda)$ so that the conclusions of Theorem \ref{maintheorem} would hold true under the condition $f\in L^{d+1-\epsilon}(Q_1)$. 

Although no proof is given in \cite{escauriaza93}, such a result is expected, provided certain building blocks of the theory are available. Those building blocks regard estimates for Green's functions associated with certain linear operators, along with well-posedness to particular parabolic problems. See, for example, \cite{chen}, \cite{escauriaza00} and \cite{grimaldi}. Of particular interest, is the following estimate:

\begin{Proposition}\label{TE1} Let $ L$ be a linear $(\lambda,\Lambda)-$elliptic operator and denote by $g(x_0, t_0, y, s) $ its Green's function in $ Q_1 $. There exist universal constants $C>0$ and $\epsilon>0$ such that, if $ p \geq (d+1) - \epsilon $ and 
\[
	\frac{1}{p} + \frac{1}{p'} = 1,
\]
the following estimate holds:
\begin{equation*}
	\int_{Q_1} [g(x_0, t_0, y, s)]^{p'}dyds  \leq C\;\;\;\;\; \forall\, (x_0,t_0), \, \in\, Q_1.
\end{equation*}
Moreover, there exists $ \beta $, universal, so that for every $ E \subset Q_r \subset Q_{1/2} $, we have
\begin{equation*}
	\left[\frac{|E|}{|Q_r|} \right]^{\beta} \int_{Q_r} g(x_0 , t_0 , y , s ) dyds \leq C \int_{E} g(x_0 , t_0 , y , s ) dyds,   \;\;\; \forall \,(x_0,t_0) \,\in \,Q_1.
\end{equation*}
\end{Proposition}

The former proposition is the parabolic variation of a result firstly obtained for the elliptic setting in \cite{fabestroock1}. In the remainder of this section, the constant $\epsilon$ appearing in Proposition \ref{TE1} will be denoted $\varepsilon_P$. When combined with additional results, Proposition \ref{TE1} yields the following Harnack inequality:
\begin{Proposition}[Harnack inequality]\label{elem1}Assume A\ref{Felliptic} holds and let $u$ be a nonnegative solution of \eqref{eq_model} in $Q_r$, for $r>0$. Then, there exists a universal constant $ C>0 $ so that
\begin{equation}
	\sup_{Q_{r/2}} u \leq C   \left[\inf_{Q_{r/2}} u + r^{2 - \frac{d+1}{q}} \| f\|_{L^{d+1-\varepsilon_P}(Q_r)} \right].
\end{equation}
\end{Proposition}
\begin{proof}
Without loss of generality we can assume $r = 1$; a linearization argument implies that $u$ solves 
\[
	u_t\,-\,Lu\,=\,f(x,t)\;\;\;\;\;\mbox{in}\;\;\;\;\;Q_1,
\]
where $L$ is a $(\lambda,\Lambda)$-elliptic operator, with measurable coefficients. 

From \cite{grimaldi}, we know that there exists a viscosity solution $v$ to
\begin{equation*} \label{P1}
\begin{cases}
	 v_{t}\,-\,Lv\,=\,f(x,t)&\;\;\;\;\;\mbox{in}\;\;\;\;\;Q_{1},\\
v\,=\,0 &\;\;\;\;\;\mbox{in}\;\;\;\;\;\partial_p Q_{1}.
\end{cases}
\end{equation*}

Also, there exists a Green's function for the operator $L$; more precisely, for all $(x,t) \in Q_1$ there exists a function $ g(x_0,t_0, \cdot,\cdot) \in L^{1 + \frac{1}{d}}(Q_1)$ such that

\begin{equation*}
	v(x,t)\,= \,\int_{Q_1} g(x,t;y,s)f(y,s) dy ds.
\end{equation*}
We have that  $w:=u-v$ is viscosity solution of the problem
\begin{equation*}
	\begin{cases}	
		w_{t} \,- \,Lw\,=\,0\;\;\;\;\;\mbox{in}\;\;\;\;\;Q_1\\
		w\,= \,0&\;\;\;\;\;\mbox{in}\;\;\;\;\;\partial_pQ_1.
	\end{cases}
\end{equation*}
Hence, the maximum principle ensures that that $u-v $ is nonnegative. By applying the Harnack's inequality for viscosity solutions (see \cite{lihewang1}) to the function $w$, it follows that
\begin{equation}\label{eq1}
	u(x,t)\;\leq\;C(u(y,s) \;-\;v(y,s)) \;+ \;v(x,t),\;\;\;\;\;\mbox{for all} \;\;\;\;\;(x,t),\,(y,s) \,\in \,Q_{1/2}.
\end{equation}
The result is consequential to \eqref{eq1}, combined with Proposition \ref{TE1}.
\end{proof}

A standard consequence of the Harnack inequality is the regularity of solutions in H\"older spaces, provided $d+1-\varepsilon_P>\frac{d+1}{2}$, as in the next lemma:

\begin{Lemma}\label{Elem2} Assume that A\ref{Felliptic} is in force and let $u$ be a viscosity solution to \eqref{eq_model}.
Then, there exist $\alpha \in (0,1)$ and  constant $C>0$, universal, so that
\begin{equation*}
	\|u\|_{C^{\alpha}(Q_{1/2})} \leq C \left[  \|u\|_{L^{\infty}(Q_1)} + \|f\|_{L^{d+1-\varepsilon_P}(Q_1)}   \right]
\end{equation*}
\end{Lemma}

Lemma \ref{Elem2} builds upon the Approximation Lemma and other elements presented in Section \ref{sec_sobolev} to yield Theorem \ref{maintheorem} under a lessened version of A\ref{fcontinuous}:

\begin{taggedtheorem}{2'}
We assume $f\in\mathcal{C}(Q_1)\cap L^{d+1-\varepsilon_P}(Q_1)$.
\end{taggedtheorem}
The number $\varepsilon_P$ in A2' will be called \textit{parabolic Escauriaza's exponent}. Besides establishing the existence of Escauriaza's exponent in the parabolic setting, Proposition \ref{elem1} also yields universal information about the H\"older exponent appearing in Lemma \ref{Elem2}. We investigate this consequence of the Harnack inequality in the next section.

\section{A universal modulus of continuity}\label{sec_um}

The statement of Lemma \ref{Elem2} acknowledges that solutions to \eqref{eq_main} are a priori in $\mathcal{C}^{\alpha,\frac{\alpha}{2}}_{loc}(Q_1)$, for $\alpha\in(0,1)$, unknown. Meanwhile, it falls short in providing a precise expression for this important quantity. 

In the sequel, methods from the realm of Geometric Tangential Analysis build upon the Harnack inequality to  provide an explicit characterization of the optimal $\alpha^*$, depending the dimension and the Escauriaza's parabolic exponent, i.e.:
\[
	\alpha^*\,=\,\alpha^*(d,\,\varepsilon_P).
\] 
We continue by presenting a general approximation lemma.

\begin{Proposition}
Let $u$ be a normalized viscosity solution to \eqref{eq_model}. Given $\delta>0$, there exists $\epsilon=\epsilon(d,\lambda,\Lambda,\delta)>0$ such that, if
\begin{equation}\label{eq_smallness}
	\left\|f\right\|_{L^{d+1-\varepsilon_P}}\,\leq\,\epsilon
\end{equation}
there exist $h\in\mathcal{C}^{1+\beta,\frac{1+\beta}{2}}_{loc}(Q_{3/4})$ and a $(\lambda,\Lambda)$-operator $\mathcal{F}$ so that 
\begin{equation}\label{hger}
	h_t\,-\,\mathcal{F}(D^2h,x,t)\,=\,0\;\;\;\;\;\mbox{in}\;\;\;\;\;Q_{3/4}
\end{equation}
and
\[
	\left\|u\,-\,h\right\|_{L^{\infty}(Q_{1/2})}\,\leq\,\delta.
\]
\end{Proposition}
\begin{proof}
We prove the proposition using a contradiction argument. We assume its statement is false. Then, there is a sequence of $(\lambda,\Lambda)$-operators $(F_n)_{n\in\mathbb{N}}$ and sequences of functions $(u_n)_{n\in\mathbb{N}}$ and $(f_n)_{n\in\mathbb{N}}$ such that
\[
	(u_n)_t\,-\,F_n(D^2u_n,x,t)\,=\,f_n(x,t)\;\;\;\;\;\mbox{in}\;\;\;\;\;Q_1,
\]
satisfying the smallness regime
\[
	\left\|f_n\right\|_{L^{d+1-\varepsilon_P}}\,\leq\,\epsilon
\]
with
\[
	\left\|u_j\,-\,h\right\|_{L^{\infty}(Q_{1/2})}\,>\,\delta_0,
\]
for any $h$ satisfying \eqref{hger} and some $\delta_0>0$. 

Because of Lemma \ref{Elem2}, we know that $u_n\to u_\infty$, through a subsequence if necessary, uniformly in compact sets of $Q_1$. Similarly, uniform ellipticity yields $F_n\to\mathcal{F}$, locally uniformly in $\mathcal{S}(d)$. These, together with the smallness regime for $f$ in $L^{d+1-\varepsilon_P}$, lead to
\[
	(u_\infty)_t\,-\,\mathcal{F}(D^2u_\infty,x,t)\,=\,0\;\;\;\;\;\mbox{in}\;\;\;\;\;Q_{3/4}.
\] 
By setting $h\equiv u_\infty$, we obtain a contradiction and conclude the proof.
\end{proof}

As before, we aim at producing an iteration argument. Its first step is the content of the next lemma.

\begin{Lemma}\label{interaction1}
Let $u$ be a normalized viscosity solution to \eqref{eq_model}. Given $\sigma\in(0,1)$, there exist $\epsilon=\epsilon(d,\lambda,\Lambda,\sigma)>0$ and $\rho=\rho(d,\lambda,\Lambda,\sigma)\in(0,1/2)$ so that, in case
\[
	\left\|f\right\|_{L^{d+1-\varepsilon_P}}\,\leq\,\epsilon,
\]
there is a constant $\zeta$ for which
\[
	\sup_{Q_\rho}\,|u(x,t)\,-\,\zeta|\,\leq\,\rho^\sigma.
\]
\end{Lemma}
\begin{proof}
Consider $\delta>0$, to be determined later. Let $h$ be the solution to the homogeneous problem governed by $\mathcal{F}$, $\delta$-close to $u$. From the standard parabolic theory (see, for example, \cite{lihewang1}), we have
\[
	\left\|h\right\|_{\mathcal{C}^{1+\beta,\frac{1+\beta}{2}}_{loc}(Q_{3/4})}\,\leq\,C,
\]
for some constant $C>0$, universal. Therefore,
\[
	\sup_{Q_r}\,|h(x,t)\,-\,h(0,0)|\,\leq\,Cr.
\]
Now, define $\rho$ and $\delta$ as
\[
	\rho\,:=\,\frac{1}{(2C)^\frac{1}{1-\sigma}}\;\;\;\;\;\;\;\;\;\;\mbox{and}\;\;\;\;\;\;\;\;\;\;\delta\,:=\,\frac{\rho^\sigma}{2};
\]
in addition, set $\zeta\,:=\,h(0,0)$. Hence,
\[
	\sup_{Q_\rho}\,|u(x,t)\,-\,\zeta|\,\leq\,\sup_{Q_\rho}\,|u(x,t)\,-\,h(x,t)|\,+\,\sup_{Q_\rho}\,|h(x,t)\,-\,\zeta|\,\leq\, \rho^\sigma,
\]
which concludes the proof.
\end{proof}

At this point, we are in the position to produce an optimal, universal, modulus of continuity for solutions to \eqref{eq_main}.

\begin{teo}[Universal modulus of continuity]\label{umc}
If $u$ is a normalized viscosity solution to \eqref{eq_model}, then, $u\in\mathcal{C}^{\alpha^*,\frac{\alpha^*}{2}}_{loc}(Q_1)$ and the following a priori estimate is satisfied:
\[
	\left\|u\right\|_{\mathcal{C}^{\alpha^*,\frac{\alpha^*}{2}}_{loc}(Q_1)}\,\leq\,C\left[\left\|u\right\|_{L^{\infty}(Q_1)}\,+\,\left\|f\right\|_{L^{{d+1-\varepsilon_P}}(Q_1)}\right],
\]
where the universal exponent is given by
\[
	\alpha^*\,=\,\alpha^*(d,\varepsilon_P)\,=\,\frac{d-2\varepsilon_P}{d+1-\varepsilon_P}.
\]
\end{teo}
\begin{proof}
Without loss of generality, we consider $u$ at the origin and assume the source term $f$ satisfies the smallness regime in \eqref{eq_smallness}. Set the exponent $\sigma$ in Lemma \ref{interaction1} as follows
\begin{equation}\label{eq_choiceofsigma}
	\sigma\,:=\,\frac{d\,-\,2\varepsilon_P}{d\,+\,1\,-\,\varepsilon_P}
\end{equation}
and let $\rho$ be the radius associated with such a choice of $\sigma$ by Lemma \ref{interaction1}. If we show the existence of a convergent sequence $(\zeta_k)_{k\in\mathbb{N}}$, so that
\begin{equation}\label{eq_optimalzeta}
	\sup_{Q_{\rho^k}}\,|u\,-\,\zeta_k|\,\leq\,\rho^{k\frac{d-2\varepsilon_P}{d+1-\varepsilon_P}},
\end{equation}
the proof is concluded. We verify \eqref{eq_optimalzeta} by induction in $k$; the step $k=1$ is precisely the content of Lemma \ref{interaction1}. Assume \eqref{eq_optimalzeta} is verified for $k=m$; we show it holds for $k=m+1$. 

Define the auxiliary function $v_m$ as follows:
\[
	v_m(x,t)\,:=\,\frac{u(\rho^\frac{m}{2}x,\rho^mt)\,-\,\zeta_m}{\rho^{m\frac{d-2\varepsilon_P}{d+1-\varepsilon_P}}}.
\]
In addition, set
\[
	F_m(M,x,t)\,:=\,\rho^{m\frac{d-2\varepsilon_P}{d+1-\varepsilon_P}}F\left(\frac{1}{\rho^{m\frac{d-2\varepsilon_P}{d+1-\varepsilon_P}}}M,x,t\right),
\]
and
\[
	f_m(x,t)\,:=\,\rho^{m\frac{d-2\varepsilon_P}{d+1-\varepsilon_P}}f(\rho^\frac{m}{2}x,\rho^mt).
\]
Notice that $v_m$ is a normalized viscosity solution to
\[
	(v_m)_t\,-\,F_m(D^2v_m,x,t)\,=\,f_m(x,t),
\]
where $f_m$ satisfies the smallness condition in \eqref{eq_smallness}, since
\[
	\int_{Q_{\rho}}|f_m(x,t)|^{d+1-\varepsilon_{P}}dxdt\,\leq\,\int_{Q_\rho}|f(x,t)|^{d+1-\varepsilon_P}dxdt\;\leq\;\epsilon.
\]
Therefore, Lemma \ref{interaction1} yields the existence of a constant $\zeta_m$ satisfying
\[
	\sup_{Q_\rho}\,|v_m\,-\,\zeta_m|\,\leq\,\rho^\frac{d-2\varepsilon_p}{d+1-\varepsilon_P}.
\]
If we define $(\zeta_m)_{m\in\mathbb{N}}$ by setting $\zeta_1=\zeta$ and 
\[
	\zeta_{m+1}\,:=\,\zeta_m\,+\,\rho^{m\frac{d-2\varepsilon_P}{d+1-\varepsilon_P}},
\]
the step $k=m+1$ in the induction process is verified.

Next, we show the sequence $(\zeta_m)_{m\in\mathbb{N}}$, as previously defined, is a Cauchy sequence of real numbers; to that end, it suffices to notice that
\begin{equation}\label{zetabla}
	|\zeta_{m}-\zeta_n|\,\leq\, C\rho^{n\frac{d-2\varepsilon_P}{d+1-\varepsilon_P}}\,\leq\,C\rho^n,
\end{equation}
for some constant $C>0$. Therefore, $\zeta_m\to\zeta_\infty\in\Rr$, as $m\to \infty$. From \eqref{eq_optimalzeta}, we have $\zeta_m\to u(0,0)$.

Because of \eqref{zetabla}, we obtain
\[
	|u(0,0)\,-\,\zeta_m|\,\leq\,\left(\frac{C}{1-\rho^{\frac{d-2\varepsilon_P}{d+1-\varepsilon_P}}}\right)\rho^{m\frac{d-2\varepsilon_P}{d+1-\varepsilon_P}}.
\]
To conclude the proof, set $r>0$ so that $\rho^{m+1}\,\leq\, r\,<\,\rho^m$; therefore,
\begin{align*}
	\sup_{K_r}\,|u(x,t)\,-\,u(y,s)|\,&\leq\,\sup_{K_r}\,|u(x,t)\,-\,\zeta_m|\,+\,\sup_{K_r}\,|u(y,s)\,-\,\zeta_m|\\
	& \leq\, \left[1+\frac{C}{1-\rho^\frac{d-2\varepsilon_P}{d+1-\varepsilon_P}}\right]\rho^{m\frac{d-2\varepsilon_P}{d+1-\varepsilon_P}}\\
	& \leq\, Cr^\frac{d_2\varepsilon_P}{d+1-\varepsilon_P},
\end{align*}
which establishes the theorem.
\end{proof}

We close this section with a few remarks.
\begin{Remark}
We observe that Theorem \ref{umc} depends only on the ellipticity of $F$ as well on the integrability of the source term.
\end{Remark}
\begin{Remark}
In \cite{jvtei}, the authors consider source terms in anisotropic Lebesgue spaces of the form $L^q(-1,0;L^P(B_1))$ and obtain expressions for the optimal $\alpha^*=\alpha^*(q,p,d)$ in several regularity regimes; this much more general framework touches our result. In particular, when
\[
	p\,=\,q\,=\,d+1-\varepsilon_P,
\]
the authors recover Theorem \ref{umc}.
\end{Remark}

\section{A priori regularity in $p$-BMO spaces}\label{sec_bmo}

In this section, we develop the regularity theory for solutions of \eqref{eq_main} in spaces of bounded mean oscillation. We denote the average of a function over $Q_\rho$ by $\langle g \rangle_{\rho}$; that is,
\[
	\langle g \rangle_{\rho} := \fint_{Q_{\rho}} g(x,t) dx dt = \frac{1}{|Q_{\rho}|} \int_{Q_{\rho}} g(x,t) dx dt.
\]

We recall that a function $g: Q_1 \to \Rr $ is said to belong to $p$-BMO if
\[
	\sup_{\rho > 0} \frac{1}{\rho^{d+1}} \int_{Q_{\rho}} |g(x,t)- \langle g \rangle |^{p} dx dt \leq C
\]
for a constant $C>0 $ independent of $ \rho $. 

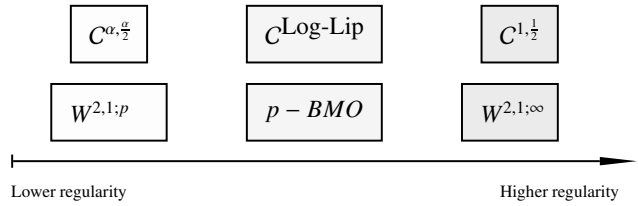
\begin{figure}[h!]\label{bmofig}
\centering
\psscalebox{.65 .65} 
{
\begin{pspicture}(0,-2.0)(12.62,2.0)
\definecolor{colour0}{rgb}{0.8039216,0.8039216,0.8039216}
\definecolor{colour1}{rgb}{0.9019608,0.9019608,0.9019608}
\definecolor{colour2}{rgb}{0.98039216,0.98039216,0.98039216}
\psframe[linecolor=black, linewidth=0.04, fillstyle=solid,fillcolor=colour0, opacity=0.39215687, dimen=outer](11.2,0.4)(9.2,-0.8)
\psframe[linecolor=black, linewidth=0.04, fillstyle=solid,fillcolor=colour1, opacity=0.39215687, dimen=outer](7.6,0.4)(4.8,-0.8)
\psframe[linecolor=black, linewidth=0.04, fillstyle=solid,fillcolor=colour2, opacity=0.39215687, dimen=outer](3.2,0.4)(0.8,-0.8)
\psframe[linecolor=black, linewidth=0.04, fillstyle=solid,fillcolor=colour0, opacity=0.39215687, dimen=outer](11.2,2.0)(9.6,0.8)
\psframe[linecolor=black, linewidth=0.04, fillstyle=solid,fillcolor=colour1, opacity=0.39215687, dimen=outer](7.6,2.0)(4.8,0.8)
\psframe[linecolor=black, linewidth=0.04, fillstyle=solid,fillcolor=colour2, opacity=0.39215687, dimen=outer](2.8,2.0)(1.2,0.8)
\rput[bl](1.6,1.2){\Large{$\mathcal{C}^{\alpha,\frac{\alpha}{2}}$}}
\rput[bl](5.2,1.2){\Large{$\mathcal{C}^{\mbox{Log-Lip}}$}}
\rput[bl](10.0,1.2){\Large{$\mathcal{C}^{1,\frac{1}{2}}$}}
\rput[bl](1.2,-0.4){\Large{$W^{2,1;p}$}}
\rput[bl](5.2,-0.4){\Large{$p-BMO$}}
\rput[bl](9.6,-0.4){\Large{$W^{2,1;\infty}$}}
\psline[linecolor=black, linewidth=0.04, tbarsize=0.07055555cm 5.0,arrowsize=0.05291667cm 3.0,arrowlength=4.4,arrowinset=0.0]{|->}(0.0,-1.2)(12.8,-1.2)
\rput[bl](0.0,-2.0){Lower regularity}
\rput[bl](10.0,-2.0){Higher regularity}
\end{pspicture}
}
\vspace{-0.0cm}
\caption{Regularity in $p-BMO$ spaces: a priori estimates in $p-BMO$ spaces, for $p>1$, bridge the gap between $W^{2,1;p}$ and $W^{2,1;\infty}$. In fact, Theorem \ref{thm1BMO} falls short in bounding $u_t$ and $D^2u$; however, it stems for gains of integrability, in comparison to $W^{2,1;p}_{loc}$-estimates. In this direction, regularity in $p-BMO$ matches - in the context of Sobolev theory - the role of Log-Lip estimates in H\"older spaces.}
\end{figure}

We work under the assumption $f\in p-BMO$, for $p>d+1-\varepsilon_P$ and A2'. We prove the following theorem:


\begin{teo}\label{thm1BMO}
Let $u\in\mathcal{C}(Q_1)$ be a normalized viscosity solution to \eqref{eq_model}. Assume that A\ref{Felliptic}, A\ref{assump_osc} and A\ref{c2arec} are in force. Then, $u_t$ and $D^{2}u$ are in $q-BMO(Q_{1/2})$ and the following a priori estimate is satisfied:
\[
	\|u_t\|_{q-BMO(Q_{1/2})}\,+\,\|D^{2}u\|_{q-BMO(Q_{1/2})} \,\leq \,C \left[ \|u\|_{L^{\infty}(Q_{1})} + \|f\|_{L^{p}(Q_1)} \right],
\]
for $q>1$.
\end{teo}

To the best of our knowledge, a priori estimates in $p-BMO$ spaces have not yet been examined in the literature, for the parabolic (fully nonlinear) setting. Besides the interest it has on its own merits, Theorem \ref{thm1BMO} also bridges the gap between the spaces $W^{2,1;p}$ and $W^{2,1;\infty}$ in a precise sense. Although regularity in $p-BMO$ does not imply boundedness either for $u_t$ or for $D^2u$, it yields improved integrability \textit{vis-a-vis} mere $p$-integrability, for every $p>1$. Before proceeding to the proof of Theorem \ref{thm1BMO}, we collect a few auxiliary results.

\begin{Lemma}\label{lemma-P} Let $G$ and $G_\infty$ be $(\lambda,\Lambda)$-elliptic operators and assume 
\[
	|G(M)\,-\,G^\infty(M)|\,+\,\left\|f\right\|_{L^p(Q_1)}\,\leq\,\overline{\epsilon},
\]
for every $M\in\mathcal{S}(d)$, where $\overline{\epsilon}>0$ is to be determined later. Moreover, suppose that $G^\infty$ has $\mathcal{C}^{2+\alpha,\frac{2+\alpha}{2}}$-a priori estimates. Then, there exist universal constants $C>0$ and $r>0$ and a second order polynomial $P$, with $\left\|P\right\|<C$ so that
\[
	\left\|u \, - \, P\right\|_{L^\infty(Q_1)}\,\leq\, r^2,
\]
where $u$ is a normalized viscosity solutions to 
\[
	u_t\,-\,G(D^2u,x,t)\,=\,f(x,t)\;\;\;\;\;\mbox{in}\;\;\;\;\;Q_1.
\]
\end{Lemma}
\begin{proof}
The proof proceeds by way of contradiction. Assume the statement is false; then, there would be sequences of $(\lambda,\Lambda)$-elliptic operators $(G_n)_{n\in\mathbb{N}}$ and $(G^\infty_n)_{n\in\mathbb{N}}$, as well as sequences of functions  $(u_n)_{n\in\mathbb{N}}$ and $(f_n)_{n\in\mathbb{N}}$ satisfying
\[
	|G_n(M,x,t)\,-\,G^\infty_n(M,x,t)|\,\sim\,1/n,
\]
for every $M\in\mathcal{S}(d)$, where $G^\infty_n$ has $\mathcal{C}^{2+\alpha,\frac{2+\alpha}{2}}$- a priori estimates, for every $n\in\mathbb{N}$. Also,
\[
	\left\|f\right\|_{L^p(Q_1)}\,\sim\,1/n,
\]
and
\[
	(u_n)_t\,-\,G_n(D^2u_n,x,t)\,=\,f_n(x,t)\;\;\;\;\;\mbox{in}\;\;\;\;\;Q_1
\]
and, every polynomial $P$ would verify
\begin{equation}\label{cont1}
	\left\|u_n\,-\,P\right\|_{L^\infty(Q_1)}\,>\,r^2,
\end{equation}
regardless of how large $n\in\mathbb{N}$ is chosen.

Because of the uniform ellipticity, $G^\infty_n$ is uniformly bounded in $K-\lipop$, where $K=K(\lambda,\Lambda)$. Therefore, through to a subsequence if necessary,
\[
	G_n^\infty \to G^\infty,
\]
globally uniformly in $\mathcal{S}(d)$. Notice that $G^\infty$ also has $\mathcal{C}^{2+\alpha,\frac{2+\alpha}{2}}$-a priori estimates. We have
\[
	|G_n(M,x,t)-G^\infty(M,x,t)|\leq|G_n(M,x,t)-G_n^\infty(M,x,t)|+|G_n^\infty(M,x,t)-G^\infty(M,x,t)|\to 0,
\]
as $n\to\infty$. Therefore, up to a subsequence, $G_n$ converges uniformly to $G^\infty$. Because $(u_n)_{n\in\mathbb{N}}$ is uniformly bounded in $\mathcal{C}^{\alpha,\frac{\alpha}{2}}_{loc}(Q_1)$, there exists $u^\infty$ so that
\[
	u_n\,\to\,u^\infty\;\;\;\;\;\mbox{in}\;\;\;\;\;\mathcal{C}^{\alpha,\frac{\alpha}{2}}_{loc}(Q_1).
\]

The stability of viscosity solutions (see Lemma \ref{stabilitylemma}) leads to
\[
	u^\infty_t\,-\,G^\infty(D^2u^\infty,x,t)\,=\,0\;\;\;\;\;\mbox{in}\;\;\;\;\;Q_1.
\]
Because $G^\infty$ has $\mathcal{C}^{2+\alpha,\frac{2+\alpha}{2}}$-estimates, $u^\infty$ is a classical solution and its Taylor's polynomial of second order $P$ is well defined; moreover, we have
\[
	\left\|u^\infty\,-\,P\right\|_{L^\infty(Q_r)}\,\leq\,Cr^{2+\alpha}.
\]
Choose $r\ll1$ in such a way that $Cr\alpha<1/2$; therefore,
\[
	\left\|u^\infty\,-\,P\right\|_{L^\infty(Q_r)}\,\leq\,\frac{r^2}{2}.
\]
Furthermore, because $u_n\to u^\infty$ uniformly in $Q_r$, we have
\[
	\left\|u_n\,-\,u^\infty\right\|_{L^\infty(Q_r)}\,\leq\,\frac{r^2}{2},
\]
for $n\gg 1$. By gathering the former inequalities, we obtain
\[
	\left\|u_n\,-\,P\right\|_{L^\infty(Q_r)}\,\leq\,r^2,
\]
which contradicts \eqref{cont1} and concludes the proof.
\end{proof}

As a corollary to Lemma \ref{lemma-P}, we have the following result:

\begin{Corollary}[Paraboloid Approximation]\label{P-Approx} Under the assumptions of Theorem \ref{thm1BMO}, there exist two universal constants, $\mu_0 >0$ and $r>0$, such that if $ u $ is a normalized solution of
\begin{equation}
	(u_\mu)_t - F_{\mu}(D^{2}u_\mu,x,t)= f(x,t) \ \ \mbox{ in } \ \  Q_1,
\end{equation}
with $ \mu + \|f\|_{L^{p}(Q_1)} \leq  \mu_0$, there exists a paraboloid $P $, with universally controlled norm $ \|P\|\leq C $ satisfying
\begin{equation*}
	\displaystyle \sup_{Q_r}|u - P | \,\leq\, r^{2}.
\end{equation*}
\end{Corollary}
\begin{proof}
The proof follows from Lemma \ref{lemma-P}, by setting $F_\mu\equiv G$ and $F^*\equiv G^\infty$, along with additional minor modifications.
\end{proof}

To establish Theorem \ref{thm1BMO}, the existence of an approximating polynomial of degree two is key. Once Corollary \ref{P-Approx} is available, we can proceed to the proof of that theorem.

\begin{proof}[Proof of Theorem \ref{thm1BMO}]

We split the proof into two steps.

\bigskip

\noindent \textbf{Step 1} We start by proving the existence of a sequence of suitable approximating polynomials. Let $u$ be a normalized viscosity solution to \eqref{eq_main} and consider $\delta_1\in(0,1)$ to be determined later. If we define $v(x,t):=\delta_1 u(x,t)$ we have that $v$ is a normalized viscosity solution of
\begin{eqnarray*}
	v_t\,-\,F_{\mu}(D^{2} v,x,t)\,= \,\tilde{f}(x,t)\;\;\;\;\;\mbox{in}\;\;\;\;\;Q_1,
\end{eqnarray*}
where $\mu:=\delta_1$ and $ \tilde{f}=  \delta_1f$. Now we choose $\delta_1$; this is set in such a way that 
\begin{eqnarray*}
	\|\tilde{f}\|_{p-BMO(Q_1)} + \mu \leq \mu_0,
\end{eqnarray*}
where $\mu_0 $ is the universal constant of Corollary \ref{P-Approx}. We prove the result for $v$, which leads to the statement of the theorem.

Our goal is to establish the existence of a sequence of polynomials $(P_k)_{k\in\mathbb{N}}$ satisfying
\[
	P_k(x,t):= \frac{1}{2} \langle A_k x , x \rangle + B_k t + \langle C_k, x \rangle + D_k,
\]
where
\[
	F^*(A_k,x,t) = \langle \tilde{f} \rangle_1 - B_k \;\;\;\;\;\mbox{and}\;\;\;\;\;\sup_{B_{r^{k}}} |v - P_k| \leq r^{2k},
\]
and
\[
r^{2(i-1)} \left( |A_i - A_{i -1}| + | B_{i} - B_{i-1}| \right) + r^{i - 1} | C_i - C_{i - 1}| + D_i \leq C r^{2(i-1)},
\]
with $r$ as in Lemma \ref{lemma-P}. We proceed by induction in $k$. Set $P_0$ and $P_{-1}$ to be
\[
	P_0(x,t)\,=\,P_{-1}(x,t)\,:=\,\frac{1}{2} \langle N x,x\rangle,
\]
where the matrix $N$ satisfies
\[
	F^*(N,x,t) = \langle \tilde{f} \rangle_1.
\]

The first step of the argument, the case $k=0$, is obviously satisfied. Suppose we have established the existence of such polynomials for $k = 0, 1, ..., i$. Then, define the re-scaled function $v_i : Q_1 \rightarrow \Rr $ by
\begin{eqnarray*}
	v_i(x,t)\,=\,\frac{(v - P_i)(r^{i}x, r^{2i}t)}{r^{2i}};
\end{eqnarray*}
the induction hypothesis ensures that $v_i$ is a normalized viscosity solutions of
\begin{eqnarray*}
	(v_i)_t\,-\,F_i(D^{2} v_i,x,t)\,=\,\tilde{f}(r^{i}x, r^{2i}t)\,=\,f_i(x,t),
\end{eqnarray*}
where
\begin{eqnarray*}
	F_i(M,x,t):= \mu F( \mu^{-1}(M + A_i),x,t) - B_i,
\end{eqnarray*}
and 
\begin{eqnarray*}
	\|f_i\|_{p-BMO(Q_r)} + \mu \leq \mu_0.
\end{eqnarray*}

In addition, because $F^*(A_i,x,t) = \langle \tilde{f}\rangle_1 - B_i $, the equation
\begin{eqnarray*}
	h_t\,-\,F^*_i(D^{2}h,x,t)\,=\,\langle \tilde{f} \rangle_1
\end{eqnarray*}
inherits $C^{2+\alpha,\frac{2+\alpha}{2}}$-estimates from the problem governed by $F^*$. Hence, Proposition \ref{P-Approx} ensures the existence of a paraboloid $ \tilde{P}$ such that
\begin{eqnarray}\label{induction1}
	\sup_{Q_r} |v_i - \tilde{P}| \leq r^{2}.
\end{eqnarray}
Therefore, by choosing
\begin{eqnarray*}
	P_{i+1}(x,t) := P_i (x,t) + r^{2i} \tilde{P}(r^{-i} x, r^{-2i} t)
\end{eqnarray*}
and rescaling $(\ref{induction1}) $ back to the unit picture, we obtain the $(i+1)-th$ step of induction. Now, we proceed to the second and final part of the proof.

\bigskip

\noindent \textbf{Step 2} 

\bigskip

\noindent Observe that
\[
	D^{2}v_m\,= \,D^2v \,- \,A_m\;\;\;\;\;\;\;\;\;\;\mbox{and}\;\;\;\;\;\;\;\;\;\;(v_m)_{t} \,= \,v_t \,-\,B_m.
\] 
Finally, choose $m$ in such a way that $ 0 < r^{m+1} < \rho \leq r^{m} $ to obtain
\begin{align*}
\frac{1}{r^{m(2 + d)}} \int_{Q_{r^{m}/2}} |D^{2}v(y,s) - A_m|^{p}dyds&+ \frac{1}{r^{m(2 + d)}} \int_{Q_{r^{m}/2}} |v_t(y,s) - B_m|^{p}dyds  \\& \leq \int_{Q_{1/2}} | D^{2}v_m|^{p}  + \int_{Q_{1/2}} | (v_m)_t |^{p}dyds \\&\leq C,
\end{align*}
where the last inequality follows from Theorem \ref{maintheorem}. This completes the proof.
\end{proof}

\bibliography{bib_2016}

\def\polhk#1{\setbox0=\hbox{#1}{\ooalign{\hidewidth\lower1.5ex\hbox{`}\hidewidth\crcr\unhbox0}}}
\begin{thebibliography}{10}

\bibitem{caffarelli89}
L.~Caffarelli.
\newblock Interior a priori estimates for solutions of fully nonlinear
  equations.
\newblock {\em Ann. of Math. (2)}, 130(1):189--213, 1989.

\bibitem{ccbook}
L.~Caffarelli and X.~Cabr{\'e}.
\newblock {\em Fully nonlinear elliptic equations}, volume~43 of {\em American
  Mathematical Society Colloquium Publications}.
\newblock American Mathematical Society, Providence, RI, 1995.

\bibitem{caffhuang}
L.~Caffarelli and Q.~Huang.
\newblock Estimates in the generalized {C}ampanato-{J}ohn-{N}irenberg spaces
  for fully nonlinear elliptic equations.
\newblock {\em Duke Math. J.}, 118(1):1--17, 2003.

\bibitem{caffstef}
L.~Caffarelli and U.~Stefanelli.
\newblock A counterexample to {$C^{2,1}$} regularity for parabolic fully
  nonlinear equations.
\newblock {\em Comm. Partial Differential Equations}, 33(7-9):1216--1234, 2008.

\bibitem{grimaldi}
M.-C. Cerutti and A.~Grimaldi.
\newblock Uniqueness for second-order parabolic equations with discontinuous
  coefficients.
\newblock {\em Ann. Mat. Pura Appl. (4)}, 186(1):147--155, 2007.

\bibitem{chen}
G.~Chen.
\newblock Non-divergence parabolic equations of second order with critical
  drift in lebesgue spaces.
\newblock {\em arXiv preprint arXiv:1511.01215}, 2015.

\bibitem{cfks98}
M.~Crandall, K.~Fok, M.~Kocan, and A.~{\'S}wi{\polhk{e}}ch.
\newblock Remarks on nonlinear uniformly parabolic equations.
\newblock {\em Indiana Univ. Math. J.}, 47(4):1293--1326, 1998.

\bibitem{CKS2000}
M.~Crandall, M.~Kocan, and A.~{\'S}wi{\polhk{e}}ch.
\newblock {$L^p$}-theory for fully nonlinear uniformly parabolic equations.
\newblock {\em Comm. Partial Differential Equations}, 25(11-12):1997--2053,
  2000.

\bibitem{jvtei}
J.-V. da~Silva and E.~Teixeira.
\newblock Sharp regularity estimates for second order fully nonlinear parabolic
  equations.
\newblock {\em Math. Ann. (2016). doi:10.1007/s00208-016-1506-y}, 2016.

\bibitem{krylov2013a}
H.~Dong, N.~Krylov, and X.~Li.
\newblock On fully nonlinear elliptic and parabolic equations with {VMO}
  coefficients in domains.
\newblock {\em Algebra i Analiz}, 24(1):53--94, 2012.

\bibitem{escauriaza93}
L.~Escauriaza.
\newblock {$W^{2,n}$} a priori estimates for solutions to fully nonlinear
  equations.
\newblock {\em Indiana Univ. Math. J.}, 42(2):413--423, 1993.

\bibitem{escauriaza00}
L.~Escauriaza.
\newblock Bounds for the fundamental solution of elliptic and parabolic
  equations in nondivergence form.
\newblock {\em Comm. Partial Differential Equations}, 25(5-6):821--845, 2000.

\bibitem{fabestroock1}
E.~Fabes and D.~Stroock.
\newblock The {$L^p$}-integrability of {G}reen's functions and fundamental
  solutions for elliptic and parabolic equations.
\newblock {\em Duke Math. J.}, 51(4):997--1016, 1984.

\bibitem{imbertsilv}
C.~Imbert and L.~Silvestre.
\newblock An introduction to fully nonlinear parabolic equations.
\newblock In {\em An introduction to the {K}\"ahler-{R}icci flow}, volume 2086
  of {\em Lecture Notes in Math.}, pages 7--88. Springer, Cham, 2013.

\bibitem{krylovabp76}
N.~Krylov.
\newblock Sequences of convex functions, and estimates of the maximum of the
  solution of a parabolic equation.
\newblock {\em Sibirsk. Mat. \v Z.}, 17(2):290--303, 478, 1976.

\bibitem{krylov84}
N.~Krylov.
\newblock Estimates for derivatives of the solutions of nonlinear parabolic
  equations.
\newblock {\em Dokl. Akad. Nauk SSSR}, 274(1):23--26, 1984.

\bibitem{krylov07b}
N.~Krylov.
\newblock Parabolic and elliptic equations with {VMO} coefficients.
\newblock {\em Comm. Partial Differential Equations}, 32(1-3):453--475, 2007.

\bibitem{krylov07a}
N.~Krylov.
\newblock Parabolic equations with {VMO} coefficients in {S}obolev spaces with
  mixed norms.
\newblock {\em J. Funct. Anal.}, 250(2):521--558, 2007.

\bibitem{krysaf2}
N.~Krylov and M.~Safonov.
\newblock A property of the solutions of parabolic equations with measurable
  coefficients.
\newblock {\em Izv. Akad. Nauk SSSR Ser. Mat.}, 44(1):161--175, 239, 1980.

\bibitem{fanghuaw2delta}
F.~Lin.
\newblock Second derivative {$L^p$}-estimates for elliptic equations of
  nondivergent type.
\newblock {\em Proc. Amer. Math. Soc.}, 96(3):447--451, 1986.

\bibitem{pimtei}
E.~Pimentel and E.~Teixeira.
\newblock Sharp {H}essian integrability estimates for nonlinear elliptic
  equations: an asymptotic approach.
\newblock {\em J. Math. Pures Appl. (9)}, 106(4):744--767, 2016.

\bibitem{silvtei}
L.~Silvestre and E.~Teixeira.
\newblock Regularity estimates for fully non linear elliptic equations which
  are asymptotically convex.
\newblock In {\em Contributions to Nonlinear Elliptic Equations and Systems},
  pages 425--438. Springer, 2015.

\bibitem{teixeirauniversal}
E.~Teixeira.
\newblock Universal moduli of continuity for solutions to fully nonlinear
  elliptic equations.
\newblock {\em Arch. Ration. Mech. Anal.}, 211(3):911--927, 2014.

\bibitem{teiurb1}
E.~Teixeira and J.-M. Urbano.
\newblock A geometric tangential approach to sharp regularity for degenerate
  evolution equations.
\newblock {\em Anal. PDE}, 7(3):733--744, 2014.

\bibitem{tso}
K.~Tso.
\newblock On an {A}leksandrov-{B}akelman type maximum principle for
  second-order parabolic equations.
\newblock {\em Comm. Partial Differential Equations}, 10(5):543--553, 1985.

\bibitem{lihewang1}
L.~Wang.
\newblock On the regularity theory of fully nonlinear parabolic equations. {I}.
\newblock {\em Comm. Pure Appl. Math.}, 45(1):27--76, 1992.

\bibitem{lihewang2}
L.~Wang.
\newblock On the regularity theory of fully nonlinear parabolic equations.
  {II}.
\newblock {\em Comm. Pure Appl. Math.}, 45(2):141--178, 1992.

\bibitem{lihewang3}
L.~Wang.
\newblock On the regularity theory of fully nonlinear parabolic equations.
  {III}.
\newblock {\em Comm. Pure Appl. Math.}, 45(3):255--262, 1992.

\bibitem{winter}
N.~Winter.
\newblock {$W^{2,p}$} and {$W^{1,p}$}-estimates at the boundary for solutions
  of fully nonlinear, uniformly elliptic equations.
\newblock {\em Z. Anal. Anwend.}, 28(2):129--164, 2009.

\end{thebibliography}
\bibliographystyle{plain}

\bigskip


\noindent\textsc{Ricardo Castillo}\\
Department of Mathematics\\
Universidade Federal de Pernambuco\\
50670-901 Recife-PE, Brazil\\
\noindent\texttt{castillo@dmat.ufpe.br}
\bigskip

\noindent\textsc{Edgard A. Pimentel (Corresponding Author)}\\
Department of Mathematics\\
Pontifical Catholic University of Rio de Janeiro -- PUC-Rio\\
22451-900, G\'avea, Rio de Janeiro-RJ, Brazil\\
\noindent\texttt{pimentel@puc-rio.br}

\end{document}